\documentclass{TPmod}
\usepackage{url}
\usepackage{setspace}
\usepackage{scrextend}

\usepackage{amsthm}
\usepackage{amssymb}
\usepackage[capitalize]{cleveref}
\usepackage{mathtools}
\usepackage{xy}
\usepackage{tikz}
\usepackage{tikz-cd}
\usetikzlibrary{shapes.geometric,arrows}

\newcommand{\Log}{\mathrm{Log}}

\renewcommand{\Sym}{\mathrm{Sym}}

\DeclareMathOperator{\Ham}{Ham}

\def\bP{\mathbb{P}}

\def\bQ{\mathbb{Q}}
\def\ZZ{\mathbb{Z}}

\newcommand{\CC}{\mathbb{C}}
\newcommand{\can}{\operatorname{can}}
\newcommand{\bfb}{\mathbf{b}}
\newcommand{\bfx}{\mathbf{x}}
\newcommand{\weak}{\operatorname{weak}}
\newcommand{\reg}{\operatorname{reg}}
\newcommand{\smooth}{\operatorname{smooth}}
\newcommand{\orb}{\mathrm{orb}}

\def\bR{\mathbb{R}}

\def\Hilb{\mathrm{Hilb}}

\def\eX{\EuScript{X}}
\def\eL{\EuScript{L}}
\def\eI{\EuScript{I}}

\newcommand{\ul}[1]{\underline{#1}}

\def\cM{\mathcal{M}}
\def\cL{\mathcal{L}}

\renewcommand{\cong}{\simeq}

\begin{document}

\thispagestyle{empty}
\title{Non-displaceable Lagrangian links in four-manifolds}
\author[Cheuk Yu Mak]{Cheuk Yu Mak}
\author[Ivan Smith]{Ivan Smith}

\address{Cheuk Yu Mak, Centre for Mathematical Sciences, University of Cambridge, Wilberforce Road, CB3 0WB, U.K.}
\address{Ivan Smith, Centre for Mathematical Sciences, University of Cambridge, Wilberforce Road, CB3 0WB, U.K.}

\begin{abstract}
Let $\omega$ denote an area form on $S^2$. 
Consider the closed symplectic 4-manifold $M=(S^2\times S^2, A\omega \oplus a \omega)$ with $0<a<A$.
We show that there are families of displaceable Lagrangian tori $\eL_{0,x},\, \eL_{1,x} \subset M$, for $x \in [0,1]$, such that 
the two-component link $\eL_{0,x} \cup \eL_{1,x}$ is non-displaceable for each $x$.
\end{abstract}
\maketitle
\section{Introduction}

\subsection{Context and results}

Let $(M,\omega_M)$ be a symplectic manifold. A Lagrangian submanifold $L \subset (M,\omega_M)$ is called (Hamiltonian) \emph{displaceable} if there exists a smooth time-dependent Hamiltonian function  $\{H_t\}=H \in C^{\infty}(M \times [0,1])$ for which the induced time-one Hamiltonian diffeomorphism $\phi^H$ satisfies $\phi^H(L) \cap L =\emptyset$.
If $L$ is not displaceable, then it is said to be (Hamiltonian) `non-displaceable'.
Understanding when a Lagrangian submanifold $L$ is displaceable  is a central question in symplectic topology; because of connections to dynamics and integrable systems, the case in which $L$ is a Lagrangian torus is especially classical.   A Lagrangian  $L$ is \emph{stably non-displaceable} if $L\times S^1 \subset M \times T^*S^1$ is non-displaceable, where $S^1\subset T^*S^1$ denotes the zero-section.

An equator $S^1_{eq} \subset S^2$ is manifestly non-displaceable, for area considerations.  In higher dimensions, a  sufficient way to prove that a Lagrangian $L$ is non-displaceable is to show that the Floer cohomology $HF(L,L) \neq 0$. 
Since Floer cohomology behaves well under taking products, it follows that $S^1_{eq}$ is stably non-displaceable, and that a product of equators in $(S^2\times S^2, A\omega \oplus A'\omega)$ is non-displaceable, for any areas $A,A' \in \R_{>0}$.

Consider now two disjoint circles $L_0, \, L_1 \subset S^2$ such that the complement of $L_0 \sqcup L_1$
comprises two discs each of area $B$ and a cylinder of area $C$. Area considerations again show that $L_0 \sqcup L_1$ is non-displaceable when $C < 2B$.  In this case, the individual $L_i \subset S^2$ \emph{are} displaceable (by rotation of the sphere), hence have vanishing Floer cohomology, and therefore $HF(L_0 \sqcup L_1, L_0 \sqcup L_1) = 0$ also vanishes.  Underscoring this, Polterovich made the remarkable observation \cite{Polterovich}  that $L_0 \sqcup L_1$ is in fact stably displaceable, i.e. the Lagrangian link $(L_0 \times S^1) \sqcup (L_1 \times S^1) \subset S^2\times T^*S^1$ is displaceable.  (The proof uses a Lagrangian suspension argument, and is  recalled in Lemma \ref{l:displace} below.) Since the Lagrangians are compact, when a displacing Hamiltonian isotopy exists, it can be chosen to be the identity outside a compact set in the cylinder factor $T^*S^1$. It follows \emph{a fortiori} that, if one fixes $A \gg 0$ sufficiently large, 
then $(L_0 \times S^1_{eq}) \sqcup (L_1 \times S^1_{eq}) \subset (S^2 \times S^2, (2B+C)\, \omega \oplus A\, \omega)$ is displaceable. 
In contrast,  whether $(L_0 \times S^1_{eq}) \sqcup (L_1 \times S^1_{eq})$ is non-displaceable or not when $A$ is small compared to $B$ and $C$ has been a long-standing open question. Surprisingly, it turns out that standard (orbifold) Floer-theoretical techniques can be applied to resolve the question, but the particular setting in which they should be applied is inspired by recent developments in tropical geometry and mirror symmetry.



Our main theorem is as follows.

\begin{thm}\label{t:nondisLag}
Let $M=(S^2 \times S^2, (2B+C)\omega \oplus (2a)\omega)$. 
Let $\eL_i=L_i \times S^1_{eq}$ in $M$, where $S^1_{eq} \subset S^2$ is the equator in the second factor. Let $\eL:=\eL_0 \sqcup \eL_1$.
If $0<a<B-C$, then for both $i=0,1$, we have $\phi(\eL_i) \cap \eL \neq  \emptyset$  for any Hamiltonian diffeomorphism $\phi$ of $M$. In particular, $\eL$ is non-displaceable. 
\end{thm}

As a consequence,

\begin{cor}
 Let $M=(S^2\times S^2, A\omega \oplus a \omega)$. If $a<A$, there are families of displaceable Lagrangian tori $\eL_{0,x},\, \eL_{1,x} \subset M$, for $x \in [0,1]$, such that 
the two-component link $\eL_{0,x} \sqcup \eL_{1,x}$ is non-displaceable for each $x$.
\end{cor}

\begin{proof}
 For any $B$ close to $A$ and $C>0$ close to $0$ such that $2B+C=2A$ and $B-C>a$, we get a non-displaceable Lagrangian link $\eL_{0} \sqcup \eL_{1} \subset (M, 2A\omega \oplus 2a\omega)$ by Theorem \ref{t:nondisLag}.
 We can vary $B$ to get a family of non-displaceable Lagrangian links. 
\end{proof}

As far as we know, this is the first example in higher dimensions (where area considerations do not pertain) of a non-displaceable Lagrangian link whose constituent components are displaceable.

\begin{figure}[ht]
\begin{center} 
\begin{tikzpicture}[scale=0.6]

\draw[semithick] ellipse  (3.05 and 3.05);

\draw[semithick] (0.7,2.95) arc (10:-10:17);
\draw[semithick] (-0.68,2.96) arc (10:-10:17);
\draw[semithick, dashed] (-0.68,2.96) arc (170:190:17);
\draw[semithick, dashed] (0.7,2.95) arc (170:190:17);
\draw[semithick] (5,-3) -- (5,3);
\draw[semithick] (8,-3) -- (8,3); 
\draw[semithick,dashed] (5,0) arc (120:60:3);
\draw[semithick] (5,0) arc (240:300:3);

\draw (0,0) node {$C$};
\draw  (2,0) node {$B$};
\draw (-2,0) node {$B$};
\draw (4,0) node {$\times$};
\draw (10,0) node {$\mathrm{vs}$};

\draw[semithick]  (15,0) ellipse  (3.05 and 3.05);
\draw[semithick] (15.7,2.95) arc (10:-10:17);
\draw[semithick] (15-0.68,2.96) arc (10:-10:17);
\draw[semithick, dashed] (15-0.68,2.96) arc (170:190:17);
\draw[semithick, dashed] (15.7,2.95) arc (170:190:17);
\draw (15,0) node {$C$};
\draw  (15+2,0) node {$B$};
\draw (15-2,0) node {$B$};
\draw (15+4,0) node {$\times$};

\draw[semithick] (22,0) ellipse (2.05 and 2.05);
\draw[semithick,dashed] (20,0) arc (120:60:4);
\draw[semithick] (20,0) arc (240:300:4);
\draw (22,1) node {$a$};
\draw (22,-1) node {$a$};

\end{tikzpicture}
\end{center}
\label{(Non)Displaceable}
\caption{Displaceable (left) but non-displaceable (right) if $0<a<B-C$}
\end{figure}
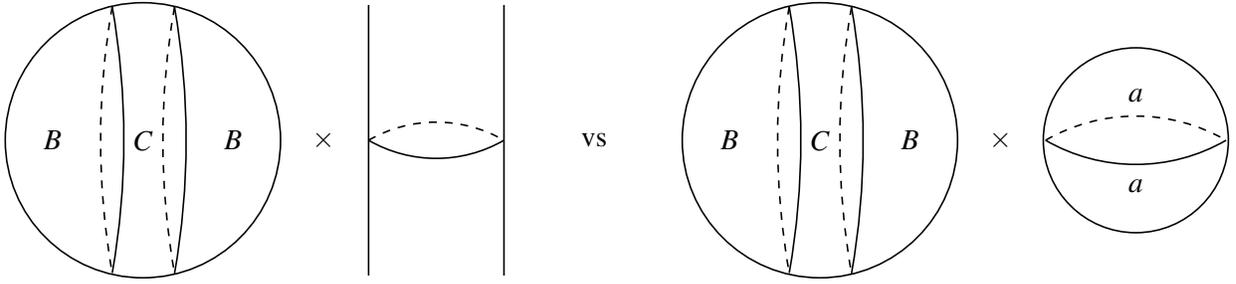

\begin{rmk} Floer theory for non-monotone symplectic manifolds such as $M$, and its Hamiltonian invariance, has been constructed in two different settings.  For rational symplectic forms, i.e. when $[\omega] \in H^2_{dR}(M)$ lies in the image of $H^2(M;\bQ)$, there is a detailed construction of the Fukaya algebra $HF^*(L,L)$ of a Lagrangian, with its $A_{\infty}$-structure and Maurer-Cartan theory, due to Charest and Woodward \cite{Charest-Woodward} based on the technique of `stabilizing divisors' due to Cieliebak and Mohnke \cite{Cieliebak-Mohnke}.
(An important structural feature of the Fukaya algebra, on which our argument relies,  is the `boundary divisor axiom'. Charest-Woodward prove a weak boundary divisor axiom but not the full one,  see Remark \ref{r:CW}.) 
For general symplectic forms, and for Floer cohomology of a pair of distinct Lagrangians, the construction relies on virtual perturbation technology.  Our proof uses rather formal properties of bulk deformation theory and bulk-deformed $A_{\infty}$-algebras for Lagrangian tori in symplectic orbifolds.   A complete development in the language of Kuranishi spaces has been given by Fukaya, Oh, Ohta and Ono \cite{FO3-1, FO3-2, FO3-construction, VFC-book};  their work was extended to the setting of orbifolds by Cho and Poddar \cite{Cho-Poddar}.  A reader uncomfortable with Kuranishi spaces should take Theorem \ref{t:superpotential} and Theorem \ref{t:superpotential2} as axioms.
\end{rmk}

\subsection{Idea of proof}\label{ss:sketch}

We will write $S^2_\alpha$ for the sphere of  area $\alpha$. 
Let $\eX:=\Sym^2(M)$ be the 2-fold symmetric product of $M$, i.e. the quotient of $M \times M$ by the $\ZZ/2$ action which exchanges the factors. 
By definition, $\eX$ is equipped with the structure of a symplectic orbifold, where the set of orbifold points is precisely the image of the diagonal.
The product Lagrangian $\eL_0 \times \eL_1$ lies away from the diagonal, so it descends to a smooth Lagrangian submanifold in $\eX$, which we denote by $\Sym(\eL)$.
Hamiltonian functions and Hamiltonian diffeomorphisms make sense in a symplectic orbifold.
Moreover, given a Hamiltonian function $H \in C^{\infty}(M \times [0,1])$
and $(z_1,z_2) \in M \times M$, the function $H(z_1,t)+H(z_2,t)$ is $\ZZ/2$-invariant and hence induces
a function $\Sym(H) \in C^{\infty}(\eX \times [0,1])$ defined by
\begin{align}
 \Sym(H)([z_1,z_2],t):=H(z_1,t)+H(z_2,t).
\end{align}
The induced Hamiltonian diffeomorphism defines a map
\begin{align}
 \Ham(M) &\to \Ham(\eX) \label{eq:HamMap} \\
 \phi &\mapsto ([z_1,z_2] \mapsto [\phi(z_1),\phi(z_2)]). \nonumber
\end{align}

\begin{rmk}
 Any Hamiltonian function on a symplectic orbifold admits a smooth lift to local uniformization charts, so any smooth lift near an orbifold point is invariant under the corresponding isotropy group. It follows that every Hamiltonian diffeomorphism of a  
 symplectic orbifold preserves the orbifold strata (this is obvious for elements in the image of \eqref{eq:HamMap}).
\end{rmk}

An immediate consequence of the  existence of the map \eqref{eq:HamMap} is that:
\begin{lem}\label{l:symlift}
 If $\Sym(\eL)$ is non-displaceable, then for both $i=1,2$ and for any Hamiltonian diffeomorphism $\phi$ of $M$, we have $\phi(\eL_i) \cap \eL \neq \emptyset$.
\end{lem}
Thus, Theorem \ref{t:nondisLag} will be a consequence of Lemma \ref{l:symlift}, via:
\begin{thm}\label{t:nondisLag2}
 Under the assumptions of Theorem \ref{t:nondisLag}, $\Sym(\eL)$ is non-displaceable.
\end{thm}

The proof of Theorem \ref{t:nondisLag2} uses a bulk deformed superpotential argument (cf. \cite{FOOObulk, Cho-Poddar}), and should be compared with the main result of \cite{FO3-S2xS2} which proved the existence of continuum families of non-displaceable Lagrangian tori $T^2 \subset (S^2\times S^2, \omega \oplus \omega)$. The calculation of the superpotential is motivated by the `tropical-holomorphic' correspondence, which relates holomorphic curves in the total space of a Lagrangian torus fibration with tropical curves in the base; the actual computation appeals to the `tautological' correspondence, which relates holomorphic discs in $\Sym(M)$ with holomorphic branched covers of discs mapping to $M$ itself.

In a little more detail, the main ideas can be summarised as follows. 
Each $\eL_i$ bounds $4$ families of Maslov $2$ discs, say in classes $\beta^j_i$, for $j=1,2,3,4$.
By (a rather trivial instance of) the tautological correspondence, the disjoint union of a disc in class 
$\beta^j_i$ and a constant map from a disc to $\eL_{1-i}$ lifts to a Maslov $2$ holomorphic disc in $\eX$ with boundary on $\Sym(\eL)$. 
Keeping track of their areas, these $8$ families of discs contribute the following terms to the superpotential of $\Sym(\eL)$:
\begin{align}
 T^a(x_2^{-1}+y_2^{-1}+x_2+y_2)+T^B(x_1+y_1)+T^{B+C}(x_1^{-1}+y_1^{-1}). \label{eq:SmoothPart}
\end{align}
The function \eqref{eq:SmoothPart} has no critical point in the units of the Novikov ring, but the Laurent polynomial 
$x_2^{-1}+y_2^{-1}+x_2+y_2$ does.

We introduce a bulk deformation such that certain holomorphic annuli in $M$ with boundary on $\eL$ contribute to the bulk deformed superpotential of $\Sym(\eL)$.
The tropical picture of one of these annuli is depicted in Figure \ref{fig:Annuli}, where 
$p_0$ and $p_1$ are
the projection of $\eL_0$ and $\eL_1$ under the $\Log$ map; they have been slightly perturbed to ease visualisation.

\begin{figure}[ht]
\begin{center}

 \includegraphics{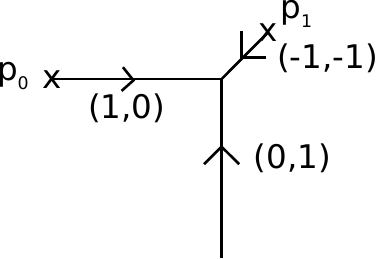}
 \caption{The tropical picture of an annulus contributing to the bulk deformed superpotential}\label{fig:Annuli}
\end{center}
\end{figure}

\begin{rmk}
 The product annulus given by the annulus bound by $L_0 \sqcup L_1$ in the $S^2_{2B+C}$ factor
 and a constant in the $S^2_{2a}$ factor is not directly helpful to prove non-displaceability of $\eL$ (for instance any argument only using such annuli would translate to $S^2_{2B+C} \times T^*S^1$).
 The annuli we use are different, and project onto a disc in the $S^2_{2a}$ factor. In particular, the annuli which contribute to the superpotential have Maslov index $2$, while the `small' visible annuli have Maslov index $0$.
\end{rmk}

\begin{rmk} There is not yet a general `tropical-holomorphic' correspondence for holomorphic curves with Lagrangian boundary conditions. Note that without some hypotheses, analytic curves in $(\CC^*)^2$ do not Gromov-Hausdorff converge to tropical curves under the rescaled logarithm maps, cf. \cite{Madani-Nisse}; for instance the curve $(z,e^z)\subset (\CC^*)^2$ has logarithmic limit set containing an interval (while the logarithmic limit set of an algebraic curve is a finite set).   In the sequel, we will give explicit constructions of the holomorphic annuli we require in Section \ref{ss:classification}, but we believe it is helpful to explain how these constructions were motivated by tropical analogues; those are given first, in Section \ref{Sec:tropical}.
\end{rmk}

After appropriate deformation, we can take the two lowest order terms of the superpotential to be 
\begin{align}
 T^a(x_2^{-1}+y_2^{-1}+x_2+y_2)+T^B(x_1+y_1+(x_1y_1)^{-1}(x_2+y_2)). \label{eq:OrPart}
\end{align}
This is possible only when $B-a-C>0$, which leads to the corresponding assumption in Theorem \ref{t:nondisLag}.
The leading order term equations for the critical points of \eqref{eq:OrPart} in the sense of \cite{FOOObulk} are 
\begin{align}\label{eq:leading}
\left\{
\begin{array}{ll}
 -x_2^{-2}+1&=0 \\
 -y_2^{-2}+1&=0 \\
 1-x_1^{-2}y_1^{-1}(x_2+y_2)&=0 \\
 1-x_1^{-1}y_1^{-2}(x_2+y_2)&=0
\end{array}
\right.
\end{align}
which admit $6$ solutions of the form $x_2=y_2 \neq 0$ and $x_1=y_1 \neq 0$.
We will indeed prove that there are at least $6$ critical points for the appropriate bulk deformed superpotential.
Combining this existence result for critical points with the general machinery developed in \cite{FO3-1,FO3-2, Cho-Poddar}, we obtain Theorem \ref{t:nondisLag2}.

\begin{remark}\label{r:Hilb}
 The Lagrangian $\Sym(\eL)$ lies in the smooth locus of $\eX$. It can therefore be lifted
 to a Lagrangian in $\Hilb^2(M)$, the Hilbert scheme of zero dimensional subschemes of length $2$ in $M$.
 Our argument can be applied to show that $\Sym(\eL)$ is non-displaceable in $\Hilb^2(M)$ when $a>B-C$ and the size of the blow-up is $a-(B-C)$ (see Remark \ref{r:Hilb2}).  
 \end{remark}

\begin{remark}\label{r:3}
 It is natural to ask for examples of Lagrangian links with more than two components. For a trivial source of examples, one can take $L_0,\ldots,L_d$ to be parallel circles in $S^2_{2B+(d-1)C}$ such that, for $i=1,\dots,d$, $L_{i-1}$ and
 $L_i$ are the respective boundary components of an area $C$ cylinder which does not meet any other $L_j$.
 The corresponding $\sqcup_{i=0}^d \eL_i$ is non-displaceable in $S^2_{2B+(d-1)C} \times S^2_a$ for appropriate $a>0$.  However, there are proper subsets of the link $\{\eL_i\}$ which are already non-displaceable.
 
 It would be more interesting to find a `Borromean' example, i.e. a non-displaceable $(d+1)$-component  Lagrangian link such that any $d$-component sublink \emph{is} displaceable.  We hope to give such examples, based on a more systematic formulation of the tropical-holomorphic correspondence, in a sequel.
\end{remark}

\subsection{Stable displaceability}

For completeness, we recall the statement and proof of Polterovich's stable displaceability result, mentioned previously. The argument presented here is a modification of Example $6.3.C$ in \cite{Polterovich}.

\begin{lem}\label{l:displace}

Let $L_0,L_1 \subset S^2_{2B+C}$  be as above.
Then $(L_0 \sqcup L_1) \times S^1$ is displaceable in $S^2_{2B+C} \times T^*S^1$, where
$S^1 \subset T^*S^1$ is the zero section.
\end{lem}

\begin{proof}
Let $M=\{x^2+y^2+z^2=R\} \subset \R^3$ with the induced round metric and area form.
Let $L_0=\{z=\eta \}$ and $L_1=\{z=-\eta \}$ for some $\eta>0$ such that the complement of 
$L_0 \sqcup L_1$ consists of two discs with area $B$ and a cylinder of area $C$.
Let $H:M \to \R$ be $H(x,y,z)=4 \pi z$.
The Hamiltonian flow of $H$ rotates the sphere, and the  induced time one Hamiltonian diffeomorphism $\phi^H_1$ is the identity arising as twice the full rotation of $M$.
The loop of Hamiltonian diffeomorphisms $(\phi^H_t)_{t \in [0,1]}$ is homotopic to the constant loop relative to the identity map, the homotopy taking place in $SO(3) \subset \Ham(M)$.
Therefore, the Lagrangian suspensions of these two Hamiltonian loops are exact Lagrangian isotopic.
For the constant loop at the identity, the Lagrangian suspension is $(L_0 \sqcup L_1) \times S^1$.
For the original loop $(\phi^H_t)_{t \in [0,1]}$, the Lagrangian suspension is
\begin{align}
\{(\phi^H_t(p),H(p),t) \in M \times \R \times S^1 \, \big| \,  p \in L_0 \sqcup L_1\}
\end{align}
where $\R \times S^1$ is identified with $T^*S^1$.
Since $H(p) \neq 0$ for all $p \in L_0 \sqcup L_1$, the two Lagrangian suspensions are disjoint.
 \end{proof}

\vspace{0.2cm}
\noindent \textbf{Acknowledgements.} This project was influenced by a set of related questions raised to the authors by  Emmanuel Opshtein.  
We are grateful to Pierrick Bousseau for discussions of his ongoing work on scattering diagrams \cite{Bousseau},  to Mohammed Abouzaid and Kaoru Ono for their interest, and to the referee for numerous helpful comments and expositional suggestions. 

The authors are partially supported by Fellowship EP / N01815X / 1 from the Engineering and Physical Sciences Research Council, U.K.

\section{Superpotentials and bulk deformation}

We summarise the theory of superpotentials for Lagrangian Floer cohomology, to establish notation which will recur later in the paper. 
More details can be found in \cite{FO3-1, FO3-2, FOOOtoric, FOOObulk, FO3-S2xS2}.

\subsection{Lagrangian Floer theory on symplectic manifolds}

Let $\cM^\circ_{k+1,l}$ be the moduli space of closed unit discs $S$ with $k+1$ boundary marked points $z_0,\dots,z_k$, ordered counterclockwise,
and $l$ interior marked points $z_1^+,\dots z_{l}^+$.
We denote the tuples $(z_0,\dots,z_{k})$ and $(z_1^+,\dots z_{l}^+)$ by $\ul{z}$ and $\ul{z}^+$, respectively.
The moduli space $\cM^\circ_{k+1,l}$ can be compactified by semi-stable nodal curves, and we denote the compactification by $\cM_{k+1,l}$.
By slight abuse of notation, a typical element in $\cM_{k+1,l}$ is also denoted by $(S,\ul{z},\ul{z}^+)$.

Let $(X,\omega_X)$ be a closed symplectic manifold, and $L\subset X$  a closed, oriented and spin Lagrangian submanifold.
For an $\omega_X$-tamed almost complex structure $J$ and a class $\beta \in H_2(X,L)$, a $J$-holomorphic stable map
with $k+1$ boundary marked points and $l$ interior marked points in class $\beta$ consists of
$((S,\ul{z},\ul{z}^+),u)$ such that
\begin{enumerate}
 \item $(S,\ul{z},\ul{z}^+) \in \cM_{k+1,l}$,
 \item $u:(S,\partial S) \to (X,L)$ is $J$-holomorphic and $u_*[S,\partial S]=\beta$,
 \item the automorphism group of $((S,\ul{z},\ul{z}^+),u)$ is finite.
\end{enumerate}
We denote the set of isomorphism classes of $((S,\ul{z},\ul{z}^+),u)$ by $\cM_{k+1,l}(X,L;J,\beta)$, or $\cM_{k+1,l}(L;\beta)$ for simplicity.

\begin{rmk}
This space is denoted by $\mathcal{M}^{\operatorname{main}}_{k+1,l}(X,L;J,\beta)$ in \cite{FO3-1, FO3-2}.\end{rmk}

There is a Kuranishi structure on $\cM_{k+1,l}(X,L;J,\beta)$ such that the evaluation maps
\begin{align}
 \ev_i:\cM_{k+1,l}(L;\beta) \to L, \quad i=0,\dots,k, \quad \ev^+_j:\cM_{k+1,l}(L;\beta) \to X, \quad j=1,\dots,l \label{eq:ev}
\end{align}
defined by $\ev_i((S,\ul{z},\ul{z}^+),u)=u(z_i)$ and $\ev^+_j((S,\ul{z},\ul{z}^+),u)=u(z_j^+)$ are weakly submersive.
We can take fiber products with smooth singular simplices $g_j:Q_j \to X$ for $j=1,\dots,l$
and $f_i:P_i \to L$ for $i=1,\dots,k$
to obtain a space $\cM_{k+1,l}(L;\beta; \ul{Q},\ul{P})$, 
where $\ul{Q}=(g_1,\dots,g_l)$ and $\ul{P}=(f_1,\dots,f_k)$.
By a multi-valued perturbation of the Kuranishi structure and a suitable triangulation of its zero locus, we obtain a singular chain
\begin{align}
 \ev_0:\cM_{k+1,l}(L;\beta; \ul{Q},\ul{P}) \to L \label{eq:vfc}
\end{align}
which is called the virtual fundamental chain and we denote it by $(\cM_{k+1,l}(L;\beta; \ul{Q},\ul{P}),\ev_0)$.

Let 
\begin{align*}
 \Lambda_0:=\left\{\sum_{i=0}^{\infty} a_i T^{\lambda_i} \ \Big | \  a_i \in \CC, 0 \le \lambda_i < \lambda_{i+1}, \lim_{i \to \infty} \lambda_i =\infty\right\}
\end{align*}
be the Novikov ring and $\Lambda_+$ be its maximal ideal.
Let $C(L;\CC)$ be the singular cochain complex\footnote{Strictly, it is necessary to work with a countably generated subcomplex of the smooth 
singular chain complex, which we grade cohomologically, i.e. by codimension.
In a slight abuse of notation, we use $e_L$ to denote both a simplicial representative of the fundamental class in $C^0(L)$, and the unit  $1 \in H^0(L)$ after we take a minimal model and pass to cohomology.}  of $L$, and
$C(L;\Lambda_0)$ be the completion of $C(L;\CC) \otimes_{\CC} \Lambda_0$ with respect to the $\R$-filtration on $\Lambda_0$.

The $A_{\infty}$ operations $m_k:C(L;\Lambda_0)^{\otimes k} \to C(L;\Lambda_0)$ are defined as follows:
\begin{align}
 m_{k,\beta}(P_1,\dots,P_k)&=(\cM_{k+1,0}(L;\beta;P_1,\dots,P_k),\ev_0)  \text{ if }\beta \neq 0  \text{ or } k\neq 0,1 \label{eq:mkbeta}\\
 m_{1,0}(P)&=\partial P \quad \text{ and }\quad m_{0,0}=0 \\
 m_{k}&=\sum_{\beta} m_{k,\beta} \otimes T^{\omega_X(\beta)}.
\end{align}
The right hand side of \eqref{eq:mkbeta} is the virtual fundamental chain 
$\ev_0:\cM_{k+1,0}(L;\beta;P_1,\dots,P_k)\to L$, which is a special case of \eqref{eq:vfc}.
These make $(C(L;\Lambda_0),\{m_k\}_{k=0}^\infty)$ a filtered $A_{\infty}$ algebra.
It is quasi-isomorphic to an $A_{\infty}$ algebra on the cohomology $(H(L;\Lambda_0),\{m_k^{\can}\}_{k=0}^\infty)$, which is called a canonical or minimal model for $(C(L;\Lambda_0),\{m_k\}_{k=0}^\infty)$.
By abuse of notation, we denote $m_k^{\can}$ by $m_k$.

The $A_{\infty}$ structure can be deformed by one or both of a choice of $\bfb \in H_*(X,\Lambda_+)$
and $b \in H^1(L,\Lambda_0)$.
We first discuss the deformation by $\bfb$. The deformation by $b$ is explained in \eqref{eq:b1} to \eqref{eq:b4}.
Given singular simplices $g_j:Q_j \to X$ for $j=1,\dots,l$, we define
\begin{align}
q'_{l,k,\beta}(\ul{Q};\ul{P})&=\frac{1}{l !} (\cM_{k+1,l}(L;\beta; \ul{Q},\ul{P}), \ev_0) \label{eq:qbetalk}
\end{align}
The right hand side of \eqref{eq:qbetalk} is $\frac{1}{l !}$ multiples of the virtual fundamental chain 
$\ev_0:\cM_{k+1,l}(L;\beta; \ul{Q},\ul{P})\to L$, which is a special case of \eqref{eq:vfc}.
By linear extension, we get a map
\begin{align}
q'_{l,k,\beta}: (C_*(X,\Lambda_+))^{\otimes l} \times C(L,\Lambda_0)^{\otimes k} \to C(L,\Lambda_0)
\end{align}
There is a subspace $E_l(C_*(X,\Lambda_+))$ of $(C_*(X,\Lambda_+))^{\otimes l}$ consisting of chains that are 
invariant under the action of the $l$-th symmetric group permuting the factors;  we denote the restriction of $q'_{l,k,\beta}$ to $E_l(C_*(X,\Lambda_+)) \times C(L,\Lambda_0)^{\otimes k}$ by $q_{l,k,\beta}$.
More explicitly, given $Q:=\sum_{j=1}^s (g_j:Q \to X) \in C_*(X,\Lambda_+)$ representing $\bfb$, 
$Q^{\otimes l}$ is a chain inside $E_l(C_*(X,\Lambda_+))$ and
we define
\begin{align}
 q_{l,k,\beta}(Q^{\otimes l};\ul{P})&=q'_{l,k,\beta}(Q^{\otimes l};\ul{P}) \label{eq:qlk0} \\
 q_{l,k}&=\sum_\beta q_{l,k,\beta} \otimes T^{\omega_X(\beta)} \label{eq:qlk}.
\end{align}
One can again use homological perturbation to pass to (co)homology, obtaining maps
\begin{align}
 H_*(X,\Lambda_+) \otimes (H(L,\Lambda_0))^{\otimes k} \to& H(L,\Lambda_0) \\
 (\bfb,x_1,\dots,x_k) \mapsto& m_k^{\bfb}(x_1,\dots,x_k):=\sum_{l=0}^{\infty} q_{l,k}(\bfb^{\otimes l},x_1,\dots,x_k) \label{eq:mbk}
\end{align}
We can define $m_{k,\beta}^{\bfb}$ similarly by replacing $q_{l,k}$ in \eqref{eq:mbk} with $q_{l,k,\beta}$.
The collection of maps $\{m_k^{\bfb}\}_{k=0}^\infty$ defines a filtered $A_{\infty}$ algebra.

Given $b \in H^1(L,\Lambda_0)$, we can write $b=b_0+b_+$  with $b_0 \in H^1(L,\CC)$ and $b_+ \in H^1(L,\Lambda_+)$.   Here we combine Poincar\'e duality on $L$ with the convention that $H(L,\Lambda_0)$ denotes the smooth singular chain complex graded cohomologically, to view $b_0$ as a singular $1$-cocycle and $b_+$ as a $\Lambda_+$-coefficient codimension one smooth singular cycle.  
For $\gamma \in H_1(L,\mathbb{Z})$, we define $\rho_{b_0}(\gamma)$ as $\exp(b_0(\gamma)) \in \mathbb{C}^*$.
Then we define
\begin{align}
 m_{k,\beta}^{\bfb, b_0}&=\rho_{b_0}(\partial \beta) m_{k,\beta}^{\bfb} \label{eq:b1}\\
 m_{k}^{\bfb, b_0}&=\sum_{\beta} m_{k,\beta}^{\bfb, b_0} \otimes T^{\omega_X(\beta)} \label{eq:b2}\\
 m_{k,\beta}^{\bfb, b}(x_1,\dots,x_k)&= \sum_{l=0}^\infty \sum_{d_0+\cdots + d_k = l} m_{k+l,\beta}^{\bfb,b_0}(\underbrace{b_+,\dots,b_+}_{d_0},x_1,\underbrace{b_+,\dots,b_+}_{d_1},x_2,\ldots,x_k,\underbrace{b_+,\dots,b_+}_{d_k}) \label{eq:b3}\\
 m_{k}^{\bfb, b}&= \sum_{\beta} m_{k,\beta}^{\bfb, b} \otimes T^{\omega_X(\beta)} \label{eq:b4}
\end{align}
Note that $m_k^{\bfb,b}=m_k^{\bfb,b'}$ if $b'-b \in H^1(L,2\pi \sqrt{-1} \ZZ)$.
The collection $\{m_k^{\bfb,b}\}_{k=0}^\infty$ also defines a filtered $A_{\infty}$ algebra.

\begin{defn}[Definition 2.3 of \cite{FO3-S2xS2}]\label{d:superpotential}
 Let $\bfb \in H(X,\Lambda_+)$ be a cycle of even codimension.
 A weak bounding cochain for $m^\bfb$ is an element $b=b_0+b_+ \in H^1(L,\Lambda_0)$ such that
 \begin{align} \label{eqn:wbc}
  \sum_{k=1}^{\infty} m_k^{\bfb,b_0}(b_+,\dots,b_+)=c\cdot e_L
 \end{align}
for some $c \in \Lambda_+$, where $e_L$ is the cohomological unit. 
 Let $\hat{\cM}_{\weak}(L,m^\bfb)$ be the set of all weak bounding cochains
 modulo the equivalence $b \sim b'$ if and only if $b'-b \in H^1(L,2\pi \sqrt{-1} \ZZ)$.
 The potential function
 $W^{\bfb}:\hat{\cM}_{\weak}(L,m^\bfb) \to \Lambda_+$ is defined by
 \begin{align}
  \sum_{k=1}^{\infty} m_k^{\bfb,b_0}(b_+,\dots,b_+)= W^{\bfb}(b)\cdot  e_L. \label{eq:Wdefn}
 \end{align}
\end{defn}

The hypothesis that $\bfb$ has even codimension in Definition \ref{d:superpotential} guarantees that the cohomological degree of
$m_k^{\bfb,b_0}(b_+,\dots,b_+)$ is even for each $k$, which is obviously a necessary condition for \eqref{eqn:wbc} to hold.
 
\begin{thm}[\cite{FO3-2, FOOOtoric, FOOObulk} for the original references, see also Theorem 2.3 of \cite{FO3-S2xS2} for a more accessible reference]\label{t:superpotential}
 Suppose that $L$ is a Lagrangian torus and 
 \begin{align}
  H^1(L,\Lambda_0)/H^1(L,2\pi \sqrt{-1} \ZZ) \subset \hat{\cM}_{\weak}(L,m^\bfb). \label{eq:H1include}
 \end{align}
 If $b \in H^1(L,\Lambda_0)$ is a critical point of the potential function 
 \begin{align}
  W^{\bfb}: H^1(L,\Lambda_0)/H^1(L,2\pi \sqrt{-1} \ZZ) \simeq (\Lambda_0 \setminus \Lambda_+)^n \to \Lambda_+ \label{eq:superpotential}
 \end{align}
 then $m_1^{\bfb,b}=0$, and hence the 
 $(\bfb,b)$-deformed Floer cohomology is isomorphic to $H(L,\Lambda_0)$.
 In particular, $L$ is non-displaceable.
\end{thm}

 In \eqref{eq:superpotential}, the identification between $H^1(L,\Lambda_0)/H^1(L,2\pi \sqrt{-1} \ZZ)=(\Lambda_0/2\pi \sqrt{-1} \ZZ)^n$
 and $(\Lambda_0 \setminus \Lambda_+)^n$ is via the exponential map.

\begin{remark}\label{r:toric}
In Section 11 of \cite{FOOObulk}, $\bfb$ is taken to lie inside the subspace generated by the
fundamental chains of the toric invariant divisors. This condition enters their argument in ensuring that \eqref{eq:H1include} is satisfied (see Proposition 2.1 of \cite{FO3-S2xS2}). It is explained in \cite{FO3-S2xS2} that, for the special case in which $L$ is a Lagrangian torus,  Theorem \ref{t:superpotential} holds as long as  \eqref{eq:H1include} holds. 
\end{remark}

For computational purposes, we recall the divisor axioms (see \cite[Lemma 7.1]{FOOObulk}, \cite[Lemma 11.8]{FOOOtoric} and \cite{FukayaCyclic}).
Fix $b \in H^1(L,\Lambda_0)$ as before. Let $\bfb$ be a chain representing a codimension $2$ cycle in in $H(X,\Lambda_+)$ that is disjoint from $L$. Then we have
\begin{align}
q_{l,k,\beta}(\bfb^{\otimes l},b_+^{\otimes k})&=\frac{(b_+(\partial \beta))^k}{k !} q_{l,0,\beta}(\bfb^{\otimes l},1)\label{eq:D1} \\
q_{l,k,\beta}(\bfb^{\otimes l},x_1,\dots,x_k)&=\frac{(\bfb \cdot \beta)^l}{l !}q_{0,k,\beta}(\mathbf{1},x_1,\dots,x_k) \label{eq:D2}
\end{align}
The geometric reason behind \eqref{eq:D1} 
is that if $u$ is a $J$-holomorphic disc in $\cM_{k+1,l}(L;\beta; \ul{Q},\ul{P})$ contributing to 
$q_{l,k,\beta}(\bfb^{\otimes l},b_+^{\otimes k})$,
then taking the fiber product between one of the evaluation maps at a boundary marked point of $\cM_{k+1,l}(L;\beta; \ul{Q},\ul{P})$ and the Poincar\'e dual of $b_+$ in $L$, the resulting moduli space will consist of 
 $J$-holomorphic discs with $k$ boundary marked points contributing to 
$q_{l,k-1,\beta}(\bfb^{\otimes l},b_+^{\otimes k-1})$.
Therefore, when the perturbation scheme in the construction of the Kuranish structures is equivariant with respect to the
cyclic group action permuting the boundary marked points, we would have
\begin{align}
q_{l,k,\beta}(\bfb^{\otimes l},b_+^{\otimes k})&=\frac{b_+(\partial \beta)}{k } q_{l,k-1,\beta}(\bfb^{\otimes l},b_+^{\otimes k-1}) \label{eq:D3}
\end{align}
The construction of such Kuranishi structures is carried out in the references above, and clearly \eqref{eq:D1} can be obtained by inductively applying \eqref{eq:D3}.
Similarly, we have 
\begin{align}
q_{l,k,\beta}(\bfb^{\otimes l},x_1,\dots,x_k)=\frac{(\bfb \cdot \beta)}{l }q_{l-1,k,\beta}(\bfb^{\otimes l-1},x_1,\dots,x_k) \label{eq:D4}
\end{align} 
because the perturbation scheme can also be made 
 equivariant with respect to the
cyclic group action permuting the interior marked points, and again \eqref{eq:D2} can be obtained by inductively applying \eqref{eq:D4}.

\begin{remark}\label{r:CW}
The boundary divisor axiom has not been proved in the Charest-Woodward setup. Instead, they prove a `weak boundary divisor axiom' in \cite[Proposition 4.33]{Charest-Woodward}, and show that is sufficient to obtain the analogue of Theorem \ref{t:superpotential} in \cite[Proposition 4.34]{Charest-Woodward} if one replaces the potential function above by the `disc potential function' $\mathcal{W}$ (defined in \cite[Equation (4.44)]{Charest-Woodward}). The disc potential function equals the potential function when the boundary divisor axiom holds (see the discussion before \cite[Lemma 4.39]{Charest-Woodward}). It is defined without boundary insertions and takes into account the local system directly. For example, in the computation below, \eqref{eq:Ww4} would be the {\it definition} of the disc potential function. 
\end{remark}

Since $q_{0,k,\beta}(\mathbf{1},x_1,\dots,x_k)$ is precisely $m_{k,\beta}(x_1,\dots,x_k)$
and $q_{0,0,\beta}(\mathbf{1},1)$ is  $m_{0,\beta}(1)$,
we have
 \begin{align}
 W^{\bfb}(b)\cdot  e_L=& \sum_{k=1}^{\infty} \sum_{\beta}T^{\omega_X(\beta)} m_{k,\beta}^{\bfb,b_0}(b_+,\dots,b_+) \label{eq:Ww1} \\
=&    \sum_{k=1}^{\infty} \sum_{\beta} T^{\omega_X(\beta)}\rho_{b_0}(\partial \beta) m_{k,\beta}^{\bfb}(b_+,\dots,b_+)  \\
=&\sum_{k=1}^{\infty} \sum_{\beta} \sum_{l=1}^{\infty} T^{\omega_X(\beta)}\rho_{b_0}(\partial \beta) \frac{(\bfb \cdot \beta)^l}{l !}q_{0,k,\beta}(\mathbf{1},b_+,\dots,b_+) \\
=& \sum_{k=1}^{\infty} \sum_{\beta} \sum_{l=1}^{\infty} T^{\omega_X(\beta)}\rho_{b_0}(\partial \beta) \frac{(\bfb \cdot \beta)^l}{l !}
\frac{(b_+(\partial \beta))^k}{k !} m_{0,\beta}(1)\\
=&\sum_{\beta} T^{\omega_X(\beta)}\exp(\bfb \cdot \beta) \exp( b(\partial \beta)) m_{0,\beta}(1) \label{eq:Ww4}
 \end{align}
where the first equality comes from the definitions \eqref{eq:Wdefn} and \eqref{eq:b2}, the second
equality comes from definition \eqref{eq:b1}, the third equality comes from \eqref{eq:D2},
the fourth equality comes from \eqref{eq:D1} (for the case $l=0$),
and the last equality comes from summing over $k$ and $l$.

If $X$ is a Fano toric manifold, $L$ is a Lagrangian torus fiber and $\beta$ is a disc class such that
$m_{0,\beta}(1)=e_L$ (this happens when $\beta$ is a `basic disc class' in the sense of \cite{FOOOtoric}), then the family of Maslov $2$ discs in class $\beta$ contributes the term
$T^{\omega_X(\beta)}\exp(\bfb \cdot \beta) \exp( b(\partial \beta))$ to $W^{\bfb}(b)$.
We can write $\partial \beta$ as $\sum_{i=1}^n a_i X_i$, where $\{X_i\}_{i=1}^n$ is a basis of $H_1(L,\mathbb{Z})$.
Letting $Y_i:=\exp(b(X_i))$, we then have the familiar formula
\begin{align}
 W^{\bfb}_L(Y_1,\dots,Y_n):=W^{\bfb}(b) =\sum_{\beta} T^{\omega_X(\beta)} \exp(\bfb \cdot \beta)  \prod_{i=1}^n Y_i^{a_i}. \label{eq:familiarW}
\end{align}

\subsection{Lagrangian Floer theory on symplectic orbifolds}\label{s:orbifold}

We need a generalization of Theorem \ref{t:superpotential} to the case that $X$ is an effective symplectic orbifold, but $L \subset X$ 
is assumed to be disjoint from the orbifold strata, i.e. $L\subset X^{reg}$ is contained in the locus of smooth points of $X$.
This generalization is carried out in \cite{Cho-Poddar}. For background on symplectic orbifolds, and in particular their symplectic forms and (contractible) spaces of compatible almost complex structures, we refer the reader to \cite{Chen-Ruan02}.
For background on orbifolds, including inertia orbifolds, orbifold morphisms, groupoids, etc.~we refer the reader to the book \cite{OrbifoldBook} and references therein.

In the orbifold case, the construction of the $A_{\infty}$-structure goes through without substantive  changes provided we consider $J$-holomorphic stable maps from  $\cM_{k+1,l}$ for which all irreducible components are {\it smooth}. 
More interestingly, in the orbifold case there are new bulk deformation directions for the superpotential coming from cycles in other components of the inertia orbifold $IX$ of $X$, as probed by $J$-holomorphic stable {\it orbifold} discs. 
This flexibility increases the chance of finding a bulk deformed superpotential that has a critical point.
We briefly recall the definition of the inertia orbifold $IX$, and how orbifold discs enter the story.

As a set, the inertia orbifold of an orbifold $X$ is
\begin{align}
IX:=\{(x,g): x \in X, g \text{ a conjugacy class in } G_x\}
\end{align}
where $G_x$ is the isotropy group of $x$.
It has an induced orbifold structure from $X$ (see \cite[Lemma 3.1.1]{Chen-Ruan}; note that $IX$ need not be connected even when $X$ is connected).
The orbifold we are going to consider in this paper is a global quotient of a smooth manifold $M$ by a finite group $G$.  In this case, there is a simpler way to describe the orbifold structure on its inertia orbifold, which we now recall (see \cite[Example 3.1.3]{Chen-Ruan}).
For each conjugacy class $g$ in $G$, we pick an element $h \in g$ and define the orbifold
\begin{align}
X_g:=M^h/C(h)
\end{align}
where $M^h \subset M$ is the set of fixed points of $h$ and $C(h)$ is the centralizer of $h$.
As suggested by the notation, up to isomorphism $X_g$ is independent of the choice of $h$ in $g$.
In particular, for $e \in G$ being the identity element, we have $M^e=M$ and $C(e)=G$ so $X_e=M/G$.
The inertia orbifold $IX$ is the disjoint union of the orbifolds
\begin{align}
IX:=\sqcup_{g} X_g
\end{align}
where $g$ runs over the conjugacy classes of $G$; each $X_g$ is called an inertia component of $IX$. We  denote $X_e$ by $X$. Note that $IX$ only depends on $X$.

An orbifold disc with $k+1$ boundary marked points and $l$ interior marked points
is a tuple $(S,\ul{z},\ul{z}^+,\ul{m})$, where  $(S,\ul{z},\ul{z}^+) \in \cM^\circ_{k+1,l}$
and $\ul{m}=(m_1,\dots,m_l)$ is a tuple of positive integers.
We equip $S$ with the unique orbifold structure 
such that the set of orbifold points is contained in $\ul{z}^+$, and
for each $j=1,\dots,l$, there is a disc neighborhood $U_j$ of $z_j^+$ which is uniformized by the branched covering map
$z \mapsto z^{m_j}$. 
The last condition means that there is an open subset $V_j \subset \mathbb{C}$
invariant with respect to the $G_{z_j^+}:=\mathbb{Z}_{m_j}$ action $z \mapsto z\exp(2\pi \sqrt{-1}/m_j)$, 
together with a surjective $G_{z_j^+}$-invariant complex analytic map $\pi: V_j \to U_j$, such that the the induced map
$\pi: V_j/G_{z_j^+} \to U_j$ is bijective. 
When $m_j=1$, $z_j^+$ is a smooth point.
For a fixed $\ul{m}$, we denote the moduli of such orbifold discs by $\cM^\circ_{k+1,l,\ul{m}}$.
It can be compactified by semi-stable nodal orbifold curves $\cM_{k+1,l,\ul{m}}$.
The union of $\cM_{k+1,l,\ul{m}}$ over all possible $\ul{m}$ is denoted by $\cM_{k+1,l}$.
By abuse of notation, a typical element in $\cM_{k+1,l}$ is also denoted by $(S,\ul{z},\ul{z}^+,\ul{m})$.

\begin{rmk}\label{rmk:good}
Under the classical differential-geometric definitions of orbifold, orbibundle etc as in \cite{Satake, Satake2},  a smooth orbifold map does not give sufficient information to define the pull-back of an orbifold vector bundle.  (This issue arises for maps into the orbifold strata $X \backslash X^{reg}$. It can be traced to the fact that orbifolds should form a 2-category rather than a 1-category. A better formulation of orbifolds as groupoids avoids these difficulties, see \cite{Lerman}; however, these are not the approaches taken by the references for Floer theory on orbifolds.)  To linearise the $\overline{\partial}$-operator at an orbifold map requires one to pull back the tangent bundle, so some restriction on the orbifold maps under consideration is essential for contructing the moduli space of $J$-holomorphic stable orbifold discs. 

The extra information required to define the pullback is a choice of a `compatible system' in the sense of \cite[Definition 4.4.1]{Chen-Ruan02}, \cite[Definition 16.12]{Cho-Poddar}.  It is possible that a smooth orbifold map has no compatible system, or has more than one isomorphism class of compatible systems.  (Two compatible systems for an orbifold map $u$ are isomorphic precisely when they give rise to isomorphic pullback orbifold bundles $u^*E$ for all orbifold bundles $E$ on the target, cf. \cite[Lemma 16.1]{Cho-Poddar}.) 
A  smooth orbifold map that has a compatible system is called {\it good}.

A compatible system comprises an indexing set $I$ and a covering of the domain and an open neighbourhood of the range of the given orbifold map by charts indexed by $I$ (in particular the charts are in bijection in domain and range), satisfying a number of compatibilities under inclusion maps and with respect to prescribed local uniformizers. The precise definition of compatible system will not play a role in the sequel, so we defer such to the references. One fact we shall need is that  a  compatible system induces a group homomorphism between the isotropy groups $G_z \to G_{u(z)}$ for $z \in S$, 
see Definition 16.11 (2)(b) of \cite{Cho-Poddar} and the subsequent paragraph, or the paragraph before \cite[Definition 4.4.1]{Chen-Ruan02}, and the 
(non-)injectivity of this homomorphism depends only on the isomorphism class of $\xi$.

Despite the differences in appearance, it turns out that the definitions  of `good maps' and `morphisms of orbifolds as groupoids' are equivalent (see \cite[Section 2.4]{OrbifoldBook}).
\end{rmk}

\begin{defn}[Definition 2.5 of \cite{Cho-Poddar}, see also Definition 2.3.3 of \cite{Chen-Ruan}]\label{d:Jmap}
Let $J$ be an almost complex structure on $X$.
A $J$-holomorphic stable map from $(S,\ul{z},\ul{z}^+,\ul{m})$ to $(X,L)$
is a pair $(u,\xi)$
such that
\begin{enumerate}
\item $u:(S,\ul{z},\ul{z}^+,\ul{m}) \to X$ is a $J$-holomorphic map (i.e. $J$-holomorphic on each irreducible component) and $u(\partial S) \subset L$;
\item $u$ is a good smooth orbifold map and $\xi$ is an isomorphism class of compatible system; 
\item the group homomorphism $G_{z_j^+} \to G_{u(z_j^+)}$ induced by $\xi$ at each orbifold point $z_j^+ \in \ul{z}^+$ is injective, where $G_{z_j^+}$ and $G_{u(z_j^+)}$ are the isotropy groups of $z_j^+$ and $u(z_j^+)$, respectively;
\item the set of $\phi:S \to S$ satisfying the following properties is finite: $\phi$ is biholomorphic, $\phi(z_i)=z_i$ for all $i$, $\phi(z_j^+)=z_j^+$ for all $j$ and $u \circ \phi=u$.
\end{enumerate}
\end{defn}

Note that if $z \in S \setminus \ul{z}^+$, $G_z$ is the trivial group and the injectivity in the third condition of Definition \ref{d:Jmap} is automatic.


In the sequel, we will encounter $J$-holomorphic orbifold discs with Lagrangian boundary conditions (lying in the regular locus) and with no bubble components. The following Lemma \ref{l:GoodMaps} will ensure that such are good.

\begin{lemma}[Lemma 4.4.11 of \cite{Chen-Ruan02}]\label{l:GoodMaps}
Let $u:S \to X$ be a smooth orbifold map.
If $S$ is irreducible and $u^{-1}(X^{\reg})$ is connected and dense, then $u$ is good with a unique choice of  isomorphism class of compatible system.
\end{lemma}

Let $IX$ be the inertia orbifold of $X$.
We denote the index set of the inertia components by $T=\{0\} \cup T'$,
and for $g \in T$, we write $X_g$ for the corresponding component, with $X_0=X$. Elements $x \in X_g$ are written as $(x,g)$.

Let $((S,\ul{z},\ul{z}^+,\ul{m}),u,\xi)$ be a $J$-holomorphic stable map. 
For each $z_j^+ \in \ul{z}^+$, $\xi$ determines a conjugacy class $g$ in $G_{u(z_j^+)}$ (see Remark \ref{r:conjugacyclass}).
We define
\begin{align}
\ev_j^+: ((S,\ul{z},\ul{z}^+,\ul{m}),u,\xi) \mapsto (u(z_j^+), g) \in IX
\end{align}
Fix a map $\bfx: \{1,\dots,l\} \to T$.
A $J$-holomorphic stable map $((S,\ul{z},\ul{z}^+,\ul{m}),u,\xi)$
is of `type' $\bfx$ if 
\begin{align}
\ev_j^+((S,\ul{z},\ul{z}^+,\ul{m}),u,\xi)  \in X_{\bfx(j)}
\end{align}
for all $j=1,\dots,l$.

\begin{remark}\label{r:conjugacyclass}
Around each $z_j^+$, we consider the local fundamental group $\pi_1(U_{z_j^+} \setminus z_j^+)$
where $U_{z_j^+} \subset S$ is a small disc neighborhood of $z_j^+$.
The isomorphism class of $\xi$ determines a group homomorphism (which is not a group isomorphism)
\begin{align}
\theta_{\xi}: \pi_1(U_{z_j^+} \setminus z_j^+) \to G_{u(z_j^+)}
\end{align}
which is well-defined up to conjugation in $G_{u(z_j^+)}$ (see \cite[Lemma 2.2.4]{Chen-Ruan02} and the paragraph before it; in the groupoid language, see the `Chen-Ruan characteristic' in \cite[Section 2.5]{OrbifoldBook}).
The conjugacy class $g$ in $G_{u(z_j^+)}$ determined by $\xi$ is the conjugacy class containing the image of the positive generator of $\pi_1(U_{z_j^+} \setminus z_j^+)$ (positive with respect to the orientation induced by the complex structure) under the map $\theta_{\xi}$.

Equivalently, if $(V_{z_j^+}, G_{z_j^+}, \pi)$ is a uniformizing chart of $U_{z_j^+}$, then the positive generator of 
$\pi_1(U_{z_j^+} \setminus z_j^+)$ defines a deck transformation of $V_{z_j^+}$. There is a unique element $h \in G_{z_j^+}$ whose action on $V_{z_j^+}$ coincides with the deck transformation.
The  conjugacy class $g$ in $G_{u(z_j^+)}$ is the conjugacy class containing the image of $h$ under the map
$G_{z_j^+} \to G_{u(z_j^+)}$ induced by $\xi$. As a result, the injectivity in Definition \ref{d:Jmap} (3), implies that elements in $g$ have order $m_j$.

\end{remark}

\begin{defn}
Let $\beta \in H_2(X,L)$.
The moduli space of isomorphism classes of 
$J$-holomorphic stable maps $((S,\ul{z},\ul{z}^+,\ul{m}),u,\xi)$ to $(X,L)$
of type $\bfx$ and such that $u_*[S,\partial S] =\beta$
is denoted by $\cM_{k+1,l}(L,J,\beta,\bfx)$. \end{defn}

\begin{rmk} This space is denoted by $\mathcal{M}^{\operatorname{main}}_{k+1,l}(L,J,\beta,\bfx)$ in \cite{Cho-Poddar}.\end{rmk}

We refer readers to \cite[Section 3, 8, 10]{Cho-Poddar} (see also \cite[Section 3]{Chen-Ruan02}) for the dimension formulae and Fredholm theory in the orbifold setting. The upshot  is that we can define $m_{k,\beta}$ and $m_k$ as before 
 to obtain a filtered $A_{\infty}$ algebra $(C(L,\Lambda_0), \{m_k\}_{k=0}^\infty)$.
As in the manifold case,
bulk deformation in the orbifold case is 
defined using 
the fiber product between $\cM_{k+1,l}(L,J,\beta,\bfx)$ and appropriate cycles under the evaluation maps $\ev_j^+$.
The crucial difference is that the codomain of $\ev_j^+$ is now the inertia orbifold $IX$, so the cycles used to cut down the image of evaluation are taken in $IX$ instead of in $X$.
For example, if $\bfx(j)=0$ for all $j$, then a stable map $((S,\ul{z},\ul{z}^+,\ul{m}),u,\xi)$ in $\cM_{k+1,l}(L,J,\beta,\bfx)$
must have $m_j=1$ for all $j$, because the local group homomorphism induced by $\xi$ is required to be injective.
This forces all the points $z_j^+$ to be smooth points, and bulk insertions by cycles in $X_{\bfx(j)}=X_0=X$ are defined as before.
If $\bfx(j) \neq 0$ for some $j$, then $J$-holomorphic orbifold discs can appear in  the moduli space $\cM_{k+1,l}(L,J,\beta,\bfx)$,
and hence contribute to the deformed filtered $A_{\infty}$ structure.

More explicitly, we define $\cM_{k+1,l, \bfx}(L;\beta; \ul{Q},\ul{P})$ to be the fiber product 
between $\cM_{k+1,l}(L,J,\beta,\bfx)$ and $\ul{Q}$, $\ul{P}$ under the evaluation maps
$\ev_i$ for $i=1,\dots,k$ and $\ev_j^+$ for $j=1,\dots,l$. Then we define (cf. \eqref{eq:qbetalk}, \eqref{eq:qlk0} and \eqref{eq:qlk}, see also \cite[Equation (12.22)]{Cho-Poddar})
\begin{align}
q'_{\beta,l,k,\bfx}(\ul{Q};\ul{P})&=\frac{1}{l !} (\cM_{k+1,l, \bfx}(L;\beta; \ul{Q},\ul{P}), \ev_0) \label{eq:qAgain}\\
q_{\beta,l,k,\bfx}(Q^{\otimes l};\ul{P})&=q'_{\beta,l,k,\bfx}(Q^{\otimes l};\ul{P}) \\
 q_{l,k}&=\sum_\beta \sum_{\bfx} q_{\beta,l,k,\bfx} \otimes T^{\omega_X(\beta)}.
\end{align}
The discussion of the deformed filtered $A_{\infty}$ structure in the previous section carries over with this new definition of $q_{l,k}$.

In this paper, the only cycle in $X_g$ for $g \neq 0$ that we will consider is the fundamental cycle $[X_g]$.
Let $H$ be the $\CC$-vector space generated by $H_*(X;\CC)$ and $[X_g]$ for all $g \neq 0$.
Let $H \otimes \Lambda_+$ denote the completion of the tensor product with respect to the $\R$-filtration.
For $\bfb \in H \otimes \Lambda_+$ and $b \in H^1(L;\Lambda_0)$, we have a filtered $A_{\infty}$ algebra
structure $\{m^{\bfb,b}_k\}_{k=1}^\infty$ on 
$H(L,\Lambda_0)$.
Weak bounding cochains for $m^\bfb$ and the bulk-deformed superpotential are defined as in Definition \ref{d:superpotential}.
Most importantly, the exact analogue of Theorem \ref{t:superpotential} holds in the orbifold setting \cite[Theorem 11.4 and 12.10]{Cho-Poddar}.

\begin{thm}\label{t:superpotential2}
 Suppose that $L\subset X^{reg}$ is a Lagrangian torus and $H^1(L,\Lambda_0)/H^1(L,2\pi \sqrt{-1} \ZZ) \subset \hat{\cM}_{\weak}(L,m^\bfb)$.
 If $b \in H^1(L,\Lambda_0)$ is a critical point of the potential function 
 \begin{align}
  W^{\bfb}: H^1(L,\Lambda_0)/H^1(L,2\pi \sqrt{-1} \ZZ) \simeq (\Lambda_0 \setminus \Lambda_+)^n \to \Lambda_+ \label{eq:superpotential2}
 \end{align}
 then $m_1^{\bfb,b}=0$, the 
 $(\bfb,b)$-deformed Floer cohomology equals $H(L,\Lambda_0)$, and  $L$ is Hamiltonian non-displaceable.
\end{thm}

The conclusion implies in particular that $L$ cannot be displaced by Hamiltonian isotopies in $X^{reg}$.

\begin{remark}
As in \cite{FO3-2}, the paper \cite{Cho-Poddar} restricts to bulk deformations by toric-invariant cycles in a toric orbifold $X$.  However, as in Remark \ref{r:toric}, their formalism applies to any bulk class provided \eqref{eq:superpotential2} and the (weak) boundary divisor axiom holds (the proof of this fact in \cite{FO3-S2xS2} uses only formal properties of filtered $A_{\infty}$-algebra and the fact that the cohomology of the torus is generated by degree one classes). 
 The main advantage of using toric-invariant bulks and a toric-invariant Lagrangian $L$, both in the usual and the orbifold case, 
 is that the inclusion $H^1(L,\Lambda_0)/H^1(L,2\pi \sqrt{-1} \ZZ) \subset \hat{\cM}_{\weak}(L,m^\bfb)$
 then holds automatically, once one has built a $T^n$-equivariant Kuranishi structure on the moduli space of discs.
 In contrast, we will give a direct proof of the inclusion $H^1(L,\Lambda_0)/H^1(L,2\pi \sqrt{-1} \ZZ) \subset \hat{\cM}_{\weak}(L,m^\bfb)$: see Lemma \ref{l:H1embed} and \ref{l:Maslov0}.

Another advantage of the toric setup is that it is easier to prove the boundary divisor axiom (see \cite[Lemma 11.8]{FOOOtoric}).
\end{remark}

\begin{remark}
The Charest-Woodward framework, establishing transversality via stabilizing divisors, has not yet been developed for symplectic orbifolds.
\end{remark}

In practice, the strategy to compute the superpotential function $W^{\bfb}$ goes as follows (see \cite[Section 12.3 and Proposition 13.1]{Cho-Poddar}). Let $\bfb=\bfb_{\smooth}+\bfb_{\orb}$, where 
$\bfb_{\smooth} \in H(X,\Lambda_+)$ represents a codimension $2$ cycle in the untwisted sector $X=X_0$ 
and $\bfb_{\orb}=\sum_{g \neq 0} a_g[X_g]$ for some $a_g \in \Lambda_+$.
By \eqref{eq:familiarW}, we can compute $W^{\bfb_{\smooth}}$ by computing $m_{0,\beta}(1)$ (i.e. the algebraic count of smooth Maslov index $2$ discs in class $\beta$) for all $\beta$.
Then we need to understand how $W^{\bfb_{\smooth}}$ changes when we consider 
$\bfb'=\bfb_{\smooth}+c[X_{g_0}]$ for some $c \in \Lambda_+$ and some twisted sector $g_0 \neq 0$.
In other words, we need to understand the contribution of $J$-holomorphic orbifold discs to  $W^{\bfb}-W^{\bfb_{\smooth}}$.

In the presence of twisted sectors, the cyclic symmetry for interior marked points breaks down: there is no divisor axiom for $\cM_{k+1,l}(L,J,\beta,\bfx)$ which allows one to reduce the number of interior orbifold points (i.e. we don't have the analogue of \eqref{eq:D2}). However, the cyclic symmetry of the boundary marked points survives and we still 
have the following analogue of \eqref{eq:D1} (see \cite[Lemma 12.7 and Definition 12.5]{Cho-Poddar}:
their $\frac{c(\beta, \mathbf{p})}{l!}PD([L])$ correspond to our $q_{l,0,\beta,\bfx}(\bfb^{\otimes l},1)$, their $r$ corresponds to our $b_+$):
\begin{align}
q_{l,k,\beta,\bfx}(\bfb^{\otimes l},b_+^{\otimes k})&=\frac{(b_+(\partial \beta))^k}{k !} q_{l,0,\beta,\bfx}(\bfb^{\otimes l},1).
\end{align}

For example, suppose that $l=2$, $\bfx(j)=g_0$ for both $j=1,2$, and 
the virtual fundamental chain satisfies
\begin{align}
(\cM_{1,2, \bfx}(L;\beta; [X_{g_0}]^{\otimes 2},1), \ev_0)=[L].
\end{align}
Then the contribution of the curves in $\cM_{1,2, \bfx}(L;\beta; [X_{g_0}]^{\otimes 2},1)$ to 
$q_{2,k,\beta,\bfx}((c[X_{g_0}])^{\otimes 2},b_+^{\otimes k})$ is
\begin{align}
\frac{(b_+(\partial \beta))^k}{k !} q_{2,0,\beta,\bfx}((c[X_{g_0}])^{\otimes 2},1)=\frac{(b_+(\partial \beta))^k}{k !} \frac{c^2}{2}e_L \label{eq:csquared1}.
\end{align}
Recall from \eqref{eq:qAgain} that there is a $\frac{1}{l!}$involved in the definition of $q_{l,k,\beta,\bfx}$, which accounts for the $2$ in the denominator in the equation above.
By summing over $k$ and taking into account the area of $\beta$, the total contribution of these curves to $W^{\bfb_{\smooth}+c[X_{g_0}]}(b)$ is
\begin{align}
T^{\omega_X(\beta)} \exp(b(\partial \beta)) \frac{c^2}{2}=T^{\omega_X(\beta)} \frac{c^2}{2}  \prod_{i=1}^n Y_i^{a_i} \label{eq:csquared2}
\end{align}
where $\partial \beta=\sum_{i=1}^n a_iX_i$ and $Y_i=\exp(b(X_i))$ as in \eqref{eq:familiarW}.

Note that the meaning of `total contribution' to $W^{\bfb_{\smooth}+c[X_{g_0}]}(b)$
means the total contribution \emph{via the terms} $q_{2,k,\beta,\bfx}((\bfb_{\smooth}+c[X_{g_0}])^{\otimes 2},b_+^{\otimes k})$ for all $k$.
It is possible (and likely) that the same underlying holomorphic curves can contribute to
$q_{l,k,\beta,\bfx}((\bfb_{\smooth}+c[X_{g_0}])^{\otimes l},b_+^{\otimes k})$ with $l>2$, 
for those $\bfx$ such that precisely two $\bfx(j)$ equal $g_0$ and the other $\bfx(j)$ equal the untwisted sector $g=0$.
Fortunately, these terms will be of higher order in $T$ (because $\bfb_{\smooth}+c[X_{g_0}] \in H(X,\Lambda_+)$) and we will not have to compute them in practice.

One can add more terms from twisted sectors to $\bfb_{\smooth}$ and the (partial) calculation of the corresponding $W^{\bfb}$ will follow the same lines as above.




\section{Classification of holomorphic orbifold discs}

The two-sphere $S^2_\alpha$ has a standard Lagrangian torus fibration arising from a Hamiltonian circle action with moment map image an interval of length $\alpha$. 
Let $M$ be the symplectic manifold in Theorem \ref{t:nondisLag}.
There is a corresponding Hamiltonian $T^2$-action on $M$ with moment map image a rectangle of side lengths $2B+C$ and $2a$; the Lagrangians $\eL_i \subset M$ of Theorem \ref{t:nondisLag} can be taken to be fibres of the Lagrangian fibration.  

We now take $X=\eX = \Sym^2(M)$, and  take $L$ to be the (symmetric) product of two distinct Lagrangian torus fibers $\eL'$ and $\eL''$ in $M$. 
In particular, we have $L \subset X^{\reg}$. The indexing set $T'$ is a singleton; we denote the unique element in $T'$ by $1$,
so that $I\eX=\eX_0 \cup \eX_1$, where $\eX_0=\eX$ and $\eX_1=M$ is the diagonal in $M \times M$. 
This section is devoted to the discussion of $J$-holomorphic stable orbifold discs mapping to $(X,L)$.

\subsection{Tautological correspondence}

We endow each $S^2$ factor of $M$ with its unique complex structure.
Let  $J_{M}$ be the induced product complex structure on $M$.
The product complex structure on $M \times M$ is invariant under the $\mathbb{Z}_2$ action
so it decends to a complex structure $J_{\eX}$ on $\eX$.
There is a well-known bijective correspondence (the `tautological' correspondence) between isomorphism classes of
 $J_{\eX}$-holomorphic maps $u:S \to \eX$ and isomorphism classes of pairs $(v,\pi_\Sigma)$, where
$\pi_{\Sigma}: \Sigma \to S$ is a $2$ to $1$ branched covering and 
$v:\Sigma \to M$ is a $J_{M}$-holomorphic map ($\Sigma$ is possibly disconnected).
Bijective correspondences of this form have been used in
 \cite{Donaldson-Smith, Smith03, OS04, Usher04, Lipshitz-cylindrical, Costello, Auroux-bordered, Mak-Smith} etc.
In \cite{Donaldson-Smith, Smith03, Usher04, Costello}, holomorphic maps from closed curves are considered, either to (families of) symmetric products of Riemann surfaces or, in \cite{Costello}, to symmetric product stacks of arbitrary-dimension smooth projective varieties. 
In \cite{OS04, Lipshitz-cylindrical, Auroux-bordered, Mak-Smith}, 
 holomorphic maps from Riemann surfaces with boundary are considered. The papers \cite{OS04, Lipshitz-cylindrical, Auroux-bordered} consider maps to the symmetric product of a Riemann surface whilst \cite{Mak-Smith}
considers maps to the symmetric product (or Hilbert scheme) of a complex surface. (However, none of the latter references endow the domain discs with an orbifold structure, or consider orbifold Floer cohomology.) 

The correspondence is defined as follows.
Given a  $J_{\eX}$-holomorphic map $u:S \to \eX$, 
we define $\Sigma$ to be the fiber product between $u$ and the quotient $M \times M \to \eX$,
so that we have the pull-back diagram
\[
\xymatrix{
\Sigma \ar[rr]^{V} \ar[d]_{\pi_{\Sigma}} & & M \times M  \ar[d]\\
S \ar[rr]_{u} &&  \eX.
}
\]
We define $v:=\pi_1 \circ V$, where $\pi_1:M \times M  \to M$ is the projection to the first factor.
Conversely, given a pair $(v,\pi_\Sigma)$, the corresponding $u$
is defined by
$u(z)=v(\pi_\Sigma^{-1}(z)) \in \eX$.

We first note that the areas of $u$ and $v$ agree: 
\begin{lemma}\label{l:Area}
Let $\omega_M$ be a symplectic form on $M$.
Let $\omega_{\eX}$ be the orbifold symplectic form on $\eX$ whose pullback to $M \times M$ is $\omega_M \oplus \omega_M$.
Then $\omega_M(v_*[\Sigma])= \omega_{\eX}(u_*[S])$, where $u$ and $v$ are related by the tautological correspondence.
\end{lemma}

\begin{proof}
It is straightforward to check that $(\omega_M \oplus \omega_M)(V_*[\Sigma])=2\omega_M(v_*[\Sigma])$, where $V:\Sigma \to M \times M$ is the $\mathbb{Z}_2$-equivariant map in the discussion above.
On the other hand, following directly from the definition of integration of an orbifold form, we also have 
\begin{align}
(\omega_M \oplus \omega_M)(V_*[\Sigma])=\int_{\Sigma} V^*  (\omega_M \oplus \omega_M)
=2 \int_S u^* \omega_{\eX}
=2\omega_{\eX}(u_*[S])
\end{align}
so the result follows.
\end{proof}

The following topological fact will be helpful when comparing the Fredholm theories for $u$ and $v$.

\begin{lemma}\label{l:bCondition}
 If $(v,\pi_\Sigma)$ is obtained tautologically from 
 a $J_{\eX}$-holomorphic map $u:S \to \eX$ with boundary on $L$,
 then 
 \begin{align}
 &\text{$\Sigma$ has $1$ or $2$ connected components. If $\Sigma$ has $2$ components, then each component is a disc.}\label{eq:bCondition}\\
 &Moreover, \text{$\partial \Sigma$ has $2$ connected components, mapped under $v$ to $\eL'$ and $\eL''$ respectively.} \nonumber
 \end{align}
\end{lemma}

\begin{proof}
  Since $\pi_{\Sigma}$ is a $2$ to $1$ branched covering, $\Sigma$ has either $1$ or $2$ connected components and $\partial \Sigma$ has either $1$ or $2$ connected components.
  Moreover, if $\Sigma$ has $2$ components, then $\pi_{\Sigma}$ is an unbranched covering so each component of $\Sigma$ is a disc.
  
  Since $u(\partial S) \subset L$, we have $V(\partial \Sigma) \subset \eL' \times \eL'' \cup \eL'' \times \eL'$.
  We know that $\eL' \times \eL''$ is disjoint from $\eL'' \times \eL'$ because $\eL' \cap \eL''=\emptyset$.
 Since $V$ is a $\mathbb{Z}_2$-equivariant map, $V(\partial \Sigma)$ must  intersect
neither or both of $\eL' \times \eL'' $ and $ \eL'' \times \eL'$; the former case is excluded because $\partial \Sigma$ is non-empty.
As a result, $\partial \Sigma$ cannot be connected so it has $2$ connected components.
 Finally, since $V$ is  a $\mathbb{Z}_2$-equivariant map, we know that the $2$ boundary components of $\Sigma$ are mapped under $V$ to $\eL' \times \eL''$ and  $\eL'' \times \eL'$, respectively.
Therefore,  the $2$ boundary components of $\Sigma$ are mapped under $v$ to $\eL'$ and $\eL''$, respectively.
\end{proof}

To compare the Fredholm theories between $u$ and $(v,\pi_{\Sigma})$, we first explain why the moduli of $S$ (i.e. $\cM_{k+1,l,\ul{m}}$) is canonically isomorphic
to a certain moduli space of double branched coverings $\pi_{\Sigma}: \Sigma \to S$.

If $S$ is a smooth disc (i.e. it is irreducible and has no orbifold point)
and $u$ is a $J_{\eX}$-holomorphic stable map in the sense of Definition \ref{d:Jmap}, then the condition that $u$ is a smooth orbifold map implies that it admits a lift $\tilde{u}: S \to M \times M$ (because a smooth orbifold map is a map that admits a smooth lift to uniformization charts, and hence a smooth lift to orbifold universal covers).
That in turn implies that $\Sigma=S \sqcup S$, $\pi_{\Sigma}:\Sigma \to S$ is the trivial 2-fold covering and
$V=\tilde{u} \sqcup (\iota \circ \tilde{u}): \Sigma \to M \times M$, where $\iota:M \times M \to M \times M$ is the involution swapping the two factors.
Thus, 
when $S$ is an orbifold disc and $u$ is a $J_{\eX}$-holomorphic stable map in the sense of Definition \ref{d:Jmap}, the 
critical values of $\pi_\Sigma$ are precisely the orbifold points of $S$.
Moreover, the corresponding $m_j$ at each orbifold point is necessarily $2$, and the images of the evaluation maps $\ev_j^+$
at these orbifold points necessarily lie in $\eX_1 \subset I\eX$ (see the last sentence of Remark \ref{r:conjugacyclass}).
The crucial point here is that the evaluation of an interior marked point $z^+_j$ of $S$ lies in $\eX_0$
if and only if $m_j=1$, and it lies in $\eX_1$ if and only if $m_j=2$.
The former case occurs if and only if $z^+_j$ is not a critical value of $\pi_{\Sigma}$
and the latter case occurs if and only if $z^+_j$ is a critical value of $\pi_{\Sigma}$. 
In short, the type $\bfx$ and $\ul{m}$ determine each other as follows: 
\begin{equation} \label{eqn:determines}
m_j= \begin{cases} 2 &  \mathrm{if} \ \bfx(j)=1 \\ 1 & \mathrm{otherwise}. \end{cases}
\end{equation}

For this fixed $\ul{m}$, and for each element $(S,\ul{z},\ul{z}^+,\ul{m})$ in $\cM_{k+1,l,\ul{m}}$, there is a unique 
marked Riemann surface $\Sigma$ (up to biholomorphism) with a 
double branched covering $\pi_{\Sigma}: \Sigma \to S$ for which the critical values of $\pi_{\Sigma}$
are exactly $\{z_j^+| m_j=2\}$ and the marked points of $\Sigma$ are precisely $\pi_{\Sigma}^{-1}(\{z_j^+| m_j=1\})$.
Moreover, if $\pi_{\Sigma'}: \Sigma' \to S$ is another double branched covering
such that  the critical values are exactly $\{z_j^+| m_j=2\}$ and the marked points of $\Sigma'$ are precisely $\pi_{\Sigma'}^{-1}(\{z_j^+| m_j=1\})$, then 
there are precisely two biholomorphisms $f_1,f_2: \Sigma \to \Sigma'$
such that $\pi_{\Sigma'} \circ f_i=\pi_{\Sigma}$. 
The biholomorphisms $f_i$ are related by the involution on $\Sigma$ induced from the double branched covering. 
This is true even if $l=0$ and $\pi_{\Sigma}$ is an unbranched covering (i.e. the requirement that $\pi_{\Sigma'} \circ f_i=\pi_{\Sigma}$ distinguishes $f_i$ from all other biholomorphisms between disjoint unions of two discs).

We denote the involution on $\Sigma$ by $\iota_{\Sigma}$ so that $f_2=f_1 \circ \iota_{\Sigma}$.
Notice that, in the tautological correspondence, 
$(v, \pi_{\Sigma})$ and $(v, \pi_{\Sigma} \circ \iota_{\Sigma})$ are in the same equivalence class in the sense that
they correspond to the same map $u(z)=v(\pi_\Sigma^{-1}(z))=v((\pi_\Sigma \circ \iota_{\Sigma})^{-1}(z))$.
Therefore, we define an isomorphism from $(\Sigma, \pi_{\Sigma})$ to $(\Sigma', \pi_{\Sigma'})$ 
to be an \emph{unordered} tuple $(f_1,f_2:=f_1 \circ \iota_{\Sigma})$ such that $f_i: \Sigma \to \Sigma'$ 
are biholomorphisms satisfying $\pi_{\Sigma'} \circ f_i=\pi_{\Sigma}$.
We then define the moduli space of covers $\cM_{k+1,l,\ul{m}}^{cover}$ as the moduli space of pairs 
$(\Sigma, \pi_{\Sigma})$ up to isomorphism such that 
$\Sigma$ is a marked Riemann surface and 
$\pi_{\Sigma}$ is a complex analytic double branched covering from $\Sigma$
to $S$, for some $S \in \cM_{k+1,l,\ul{m}}$, such that $\pi_{\Sigma}^{-1}(\{z_j^+| m_j=1\})$ is precisely the set of marked points of $\Sigma$.
By the discussion above, we conclude that the moduli of covers $\cM_{k+1,l,\ul{m}}^{cover}$
is canonically isomorphic to $\cM_{k+1,l,\ul{m}}$, which has dimension 
\begin{align}
\dim(\cM_{k+1,l,\ul{m}}^{cover})=\dim(\cM_{k+1,l,\ul{m}})=k+1+2l-3. \label{eq:DimDomain}
\end{align}

\begin{remark}\label{r:moduliexplain}
When the right hand side of \eqref{eq:DimDomain} is negative, it should be interpreted as the dimension of the moduli space 
minus the dimension of the generic automorphism group.
The bijective correspondence between $\cM_{k+1,l,\ul{m}}^{cover}$ and $\cM_{k+1,l,\ul{m}}$ is compatible with automorphism groups in the following sense.
Let $S \in \cM_{k+1,l,\ul{m}}$ be a Riemann surface with non-trivial automorphism group (this occurs when $l=0$ and $k \in \{0,1\}$).
For each biholomorphism $g:S \to S$, there are precisely two
biholomorphisms $f_1,f_2: \Sigma \to \Sigma$ such that $\pi_{\Sigma} \circ f_i=g \circ \pi_{\Sigma}$.
Moreover, we have $f_2= f_1 \circ \iota_{\Sigma}$.
Therefore $g$ uniquely determines an isomorphism $(f_1,f_2)$ from $(\Sigma,\pi_{\Sigma})$ to itself.
Conversely, an isomorphism from $(\Sigma,\pi_{\Sigma})$ to itself uniquely determines an automorphism of $S$.
Therefore, there is a canonical bijective correspondence between the automorphism group of
 $(\Sigma,\pi_{\Sigma})$ and the automorphism group of $S$.
\end{remark}

\begin{remark}
If $(\Sigma, \pi_{\Sigma}), (\Sigma', \pi_{\Sigma'}) \in \cM_{k+1,l,\ul{m}}^{cover}$, then $\Sigma$ is homeomorphic to
$\Sigma'$.
From the discussion above, $(\Sigma, \pi_{\Sigma})$ is isomorphic to $(\Sigma', \pi_{\Sigma'})$
if and only if $\Sigma$ is biholomorphic to $\Sigma'$ as marked Riemann surfaces.
Therefore, $\cM_{k+1,l,\ul{m}}^{cover}$ is canonically a subspace of the moduli space of complex structures on the underlying marked topological surface of $\Sigma$ (if $\Sigma$ is disconnected, the subspace lies inside the locus in which the complex structures on the two connected components agree). It is usually a proper subspace because most Riemann surfaces do not have a biholomorphic involution. However, as we will see and use, every annulus has a biholomorphic involution, which gives us
an isomorphism from the moduli space of  annuli to the moduli space of discs with two interior marked points (see Remark \ref{r:cylLift}).
\end{remark}

Let $S \in \cM_{k+1,l,\ul{m}}$ and $(\Sigma, \pi_{\Sigma}) \in \cM_{k+1,l,\ul{m}}^{cover}$ be as before.
The space of smooth orbifold maps
 \[
 C^{\infty}((S,\partial S), (\eX,L))
 \]
can be identified with the space of $\mathbb{Z}_2$-equivariant smooth maps
\[
C^{\infty}((\Sigma, \partial \Sigma), (M \times M, \cL' \times \cL'' \cup \cL'' \times \cL'))
\] 
modulo the relation $f \sim f \circ \iota_{\Sigma}$. 
Indeed, we can take $(\Sigma, \mathbb{Z}_2, \pi_{\Sigma})$ as a uniformizing chart for $S$
and $(M \times M, \mathbb{Z}_2, \pi_{M\times M})$ as a uniformizing chart for $\eX$, where $\pi_{M\times M}: M\times M \to \eX$ is the quotient map.
By projecting to the first factor of $M \times M$ (as in the tautological correspondence), the space of $\mathbb{Z}_2$-equivariant smooth maps
$C^{\infty}((\Sigma, \partial \Sigma), (M \times M, \cL' \times \cL'' \cup \cL'' \times \cL'))$
modulo the relation $f \sim f \circ \iota_{\Sigma}$
is isomorphic to $C^{\infty}((\Sigma, \partial \Sigma), (M,\cL' \sqcup \cL''))$ modulo the relation $f \sim f \circ \iota_{\Sigma}$.
Notice that the maps $f$ and $f \circ \iota_{\Sigma}$
lie in different connected components of $C^{\infty}((\Sigma, \partial \Sigma), (M,\cL' \sqcup \cL''))$ (because if $f(\partial^1 \Sigma) \subset \cL'$, then $f \circ \iota_{\Sigma}(\partial^1 \Sigma) \subset \cL''$, where $\partial^1 \Sigma$ is a connected component of $\partial \Sigma$).
Therefore, the involution $f \mapsto f \circ \iota_{\Sigma}$ is free.

We can use the same reasoning to compare the  Fredholm regularities of  a $J_{\eX}$ holomorphic curve $u$
and the corresponding $J_{M}$ holomorphic curve $v$.
More precisely, for any $J_{\eX}$ holomorphic stable map $u:(S,\partial S) \to (\eX,L)$ (in the sense of Definition \ref{d:Jmap}), 
the function spaces 
\[
W^{1,p}((S,\partial S), (u^*T\eX,u|_{\partial S}^*TL)) \quad \mathrm{and} \quad 
L^{p}((S,\partial S), (u^*T\eX,u|_{\partial S}^*TL)\otimes \Lambda^{0,1})
\]
 are canonically identified with the corresponding function spaces for the pairs $((\Sigma, \partial \Sigma), (M,\cL' \sqcup \cL"))$ modulo the relation $- \sim - \circ \iota_{\Sigma}$.
This identification further identifies the Fredholm sections associated to $(S,u)$ and $(\Sigma,v)$, and hence 
$(S,u)$ is regular if and only if $(\Sigma,v)$ is regular.

We can summarize this discussion as follows.
Let $\cM'(\Sigma, \cL' \sqcup \cL'', J_M, \beta)$ be the moduli space of
$J_M$-holomorphic maps $v: \Sigma \to M$ such that the two boundary components of $\Sigma$
are mapped to different connected components of $\cL' \sqcup \cL''$,
and for which $u(z):=v(\pi_{\Sigma}^{-1}(z))$ is in the homotopy class $\beta$.
There is a free involution on $\cM'(\Sigma, \cL' \sqcup \cL'', J_M, \beta)$, given by $v \mapsto v \circ \iota_{\Sigma}$; 
we denote the quotient by $\cM(\Sigma, \cL' \sqcup \cL'', J_M, \beta)$.
We define $\cM_{k+1,l,\ul{m}}(\cL' \sqcup \cL'', J_M, \beta)$ to be
the moduli space of $(\Sigma,v,\pi_{\Sigma})$ such that
$(\Sigma, \pi_{\Sigma}) \in \cM_{k+1,l,\ul{m}}^{cover}$
and $v \in \cM(\Sigma, \cL' \sqcup \cL'', J_M, \beta)$.
Then: 


\begin{lemma}\label{l:FredReg}
The tautological correspondence defines a canonical isomorphism 
\[
\cM_{k+1,l}(L,J_{\eX},\beta,\bfx) \ \cong \ \cM_{k+1,l,\ul{m}}(\cL' \sqcup \cL'', J_M, \beta).
\]
Moreover, this isomorphism takes regular  elements to regular elements.
\end{lemma}

%

\begin{proof}
The data $\bfx$ is equivalent to the data of $\ul{m}$ via \eqref{eqn:determines}.  
We can define a map from $\cM_{k+1,l,\ul{m}}(\cL' \sqcup \cL'', J_{\eX}, \beta)$ to $\cM_{k+1,l}(L,J_M,\beta,\bfx)$ which sends
 $(\Sigma,v,\pi_{\Sigma})$ to
$u(z):=v(\pi_{\Sigma}^{-1}(z))$.
This map is surjective by the paragraph before Lemma \ref{l:Area}.
This map is also injective because we have quotiented out the involution induced from $\iota_{\Sigma}$ in both the definition of 
$\cM_{k+1,l,\ul{m}}(\cL' \sqcup \cL'', J_M, \beta)$ and in the definition of isomorphism of elements in $\cM_{k+1,l,\ul{m}}^{cover}$.
The regularity was discussed above.
\end{proof}


\begin{remark}\label{r:Mas}
A natural question from the point of view of Floer theory is how the Maslov index of $u$ and $v$ are related.
The definition of the Maslov index of a holomorphic map from an 
orbifold Riemann surface can be found in \cite[Section 3]{Cho-Poddar}.
In our case, we bypass the calculation of the Maslov index of $u$ because
the tautological correspondence allows us to identify the moduli of $u$ 
with the moduli of $(v,\pi_{\Sigma})$, which in turn allow us to compute the virtual dimension using $(v,\pi_{\Sigma})$.
\end{remark}

\begin{remark}
By Lemma \ref{l:bCondition}, there are two types of $(v, \pi_{\Sigma})$, determined by whether $\Sigma$ is connected or not.
We know that $\Sigma$ is disconnected if and only if $\bfx$ involves only the untwisted sector.
Therefore, when we study the filtered $A_{\infty}$ structure $m_{k,\beta}$
on $H(L,\Lambda_0)$,
it is sufficient to consider the moduli spaces of $J_{M}$-holomorphic maps
from $\Sigma:=(S,\ul{z}) \sqcup (S,\ul{z})$ to $M$, where $(S,\ul{z}) $ is a semi-stable nodal disc with
$k+1$ boundary marked points and no interior marked points.
In contrast, if we introduce a bulk deformation $\bfb=\bfb_{\smooth}+\bfb_{\orb} [\eX_1]$ 
such that $\bfb_{\smooth} \in H_*(\eX,\Lambda_+)$ and $\bfb_{\orb} \in \Lambda_+$, then 
the construction of the maps $m_{k,\beta}^{\bfb}$
will involve additional $J_{\eX}$ holomorphic stable maps $((S,\ul{z},\ul{z}^+,\ul{m}),u,\xi)$.
Moreover, for each $((S,\ul{z},\ul{z}^+,\ul{m}),u,\xi)$ contributing to $m_{k,\beta}^{\bfb}$, the coefficient of the term 
it contributes has a factor  $\bfb_{\orb}^l$, where $l$
is the number of orbifold points of $S$ (cf. \eqref{eq:csquared1}, \eqref{eq:csquared2}).
\end{remark}

We recall the Riemann-Hurwitz and virtual dimension formulae.

\begin{lemma}[see Section 7.2.1 of \cite{DonaldsonRS}]\label{l:RH}
Let $\pi_{\Sigma}:\Sigma \to S$ be a $k$-fold branched covering between compact Riemann surfaces (possibly with non-empty boundary) 
such that all critical points are in the interior of $\Sigma$.
Then 
\begin{align}
\chi(\Sigma)=k\chi(S)-\sum_{p \in \Sigma}(e_p-1)
\end{align}
where $\chi(-)$ denotes the  Euler characteristic and $e_p$ is the ramification index at $p$.
In particular, if $k=2$
\begin{align}
\chi(\Sigma)=2\chi(S)-|\operatorname{crit}(\pi_{\Sigma})|
\end{align}
where $|\operatorname{crit}(\pi_{\Sigma})|$ is the number of critical points of $\pi_{\Sigma}$.
\end{lemma}

\begin{lemma}[see Theorem C.1.10(ii) of \cite{McDuffSalamonBook}]\label{l:vir}
Let $\Sigma$ be a (smooth) compact Riemann surface with boundary.
Let $J$ be an $\omega$-tamed almost complex structure on a symplectic manifold $(M^{2n},\omega)$ of dimension $2n$.
The moduli space of $J$-holomorphic maps $v:\Sigma \to M$ with fixed Lagrangian boundary conditions has virtual dimension
$n(\chi(\Sigma))+\mu(v)$
where $\chi(\Sigma)$ is the Euler characteristic of $\Sigma$ and $\mu(v)$ is the Maslov index of $v$.

More generally, if $\Sigma$ is allowed to vary in a moduli space of Riemann surfaces of dimension\footnote{The dimension $d$ should be interpreted as the dimension of the moduli space minus the dimension of the automorphism group.} $d$, then the moduli space of $J$-holomorphic maps $v:\Sigma \to M$ with fixed Lagrangian boundary conditions has virtual dimension
$n(\chi(\Sigma))+\mu(v)+d$.

\end{lemma}

\begin{cor}
Let $ v\in \cM(\Sigma, \cL' \sqcup \cL'', J_M, \beta)$ and $\beta_{M}:=v_*[\Sigma, \partial \Sigma] \in H_2(M,\eL' \sqcup \eL'')$.
The virtual dimension of $\cM(\Sigma, \cL' \sqcup \cL'', J_M, \beta)$ is given by
\begin{align}
2(\chi(\Sigma))+\mu(\beta_{M})=2(2-| \{j \, | \, m_j=2  \} |)+\mu(\beta_{M}). \label{eq:DimMap}
\end{align}

Moreover, the virtual dimension of $\cM_{k+1,l,\ul{m}}(\cL' \sqcup \cL'', J_M, \beta)$ (and hence the virtual dimension of $\cM_{k+1,l}(L,J,\beta,\bfx)$) is given by
\begin{align}
\dim(\cM_{k+1,l}(L,J,\beta,\bfx))&=\dim(\cM_{k+1,l,\ul{m}}(\cL' \sqcup \cL'', J_M, \beta))\\
&=4+(k+1)+2 | \{j\, | \, m_j=1  \} |+\mu(\beta_{M}) -3. \label{eq:DimTotal}
\end{align}
\end{cor}

\begin{proof}
The left hand side of \eqref{eq:DimMap} is a direct application of Lemma \ref{l:vir}
and the right hand side of \eqref{eq:DimMap} is obtained from applying Lemma \ref{l:RH}.
Finally, \eqref{eq:DimTotal} is a direct application of Lemma \ref{l:vir} by adding  \eqref{eq:DimDomain} with \eqref{eq:DimMap}.
\end{proof}


\begin{rmk}
When $\bfx \equiv 0$, formula \eqref{eq:DimTotal}
reduces to $\dim(L)+(k+1)+2l+\mu(\beta)-3$, noting $\mu(\beta_{M})=\mu(\beta)$ when the domain is a smooth disc.
This recovers the virtual dimension formula for $J_{\eX}$-holomorphic maps
from a disc with $k+1$ boundary marked points and $l$ interior marked points in class $\beta$.
\end{rmk}

\begin{lem}\label{l:H1embed}
Let $\bfb=\bfb_{\smooth}+\bfb_{\orb}[\eX_1]$, and suppose $\bfb_{\smooth} \in H_{6}(\eX,\Lambda_+)$.
Suppose that there is no non-constant $J_{\eX}$ holomorphic stable map $((S,\ul{z},\ul{z}^+,\ul{m}),u,\xi)$ to $(\eX,L)$
such that the corresponding $v_*[\Sigma, \partial \Sigma]$ has vanishing  Maslov index.
Then $H^1(L,\Lambda_0)/H^1(L,2 \pi \sqrt{-1} \ZZ)$
is contained in $\hat{\cM}_{\weak}(L,m^{\bfb})$.

Moreover, given $b \in H^1(L,\Lambda_0)/H^1(L,2 \pi \sqrt{-1} \ZZ)$, every $((S,\ul{z},\ul{z}^+,\ul{m}),u,\xi)$ that contributes to $W^{\bfb}(b)$  has $\mu(v_*[\Sigma, \partial \Sigma])=2$.
\end{lem}

\begin{proof}
Let $\beta \in H_2(\eX,L)$ and suppose  that $((S,\ul{z},\ul{z}^+,\ul{m}),u,\xi)$ contributes to $m_{k,\beta}^{\bfb,b_0}(b_+,\dots,b_+)$.
Let $\beta_{M}=v_*[\Sigma, \partial \Sigma]$.
Since $\bfb_{\smooth}$ is a codimension $2$ cycle in $\eX$ and $[\eX_1]$ is codimension $0$ in its inertia component,
after taking fiber product between $\cM_{k+1,l}(L,J,\beta,\bfx)$ and these cycles at interior marked points, the virtual dimension becomes
\begin{align}
&(4+(k+1)+2 | \{j\, | \, m_j=1  \} |+\mu(\beta_{M})-3)-2 | \{j\, | \, m_j=1  \} |-0 \label{eq:caldim}\\
=&4+(k+1)+\mu(\beta_{M})-3 \label{eq:caldim2}
\end{align}
where the first term of \eqref{eq:caldim} is the dimension of $\cM_{k+1,l}(L,J,\beta,\bfx)$ in \eqref{eq:DimTotal},
the second term of \eqref{eq:caldim} comes from the codimension of $\bfb_{\smooth}$ 
(taking the fiber product between an evaluation map and a codimension $d$ cycle drops the vitrual dimension by $d$) 
and the last term of \eqref{eq:caldim} comes from the codimension of $[\eX_1]$ which is $0$. 

Since $b \in H^1(L,\Lambda_0)$, the cohomological degree of $m_{k,\beta}^{\bfb,b_0}(b_+,\dots,b_+)$ is given by
\begin{align}
&\dim(L)-((4+(k+1)+\mu(\beta_{M})-3)-k)\\
=&2-\mu(\beta_{M}) \label{eq:virmbbb}.
\end{align}
This is because the cohomological degree is given by subtracting from the dimension of $L$  the virtual dimension of
the fiber product between the evaluation maps of $\cM_{k+1,l}(L,J,\beta,\bfx)$
with the cycles $\bfb_{\smooth}$, $[\eX_1]$, and  the Poincar\'e dual of $b$, at all the interior marked points with $m_j=1$, all the interior marked points with $m_j=2$, and the boundary marked points $z_1,\dots,z_k$, respectively.
The virtual dimension of
this fiber product is nothing but \eqref{eq:caldim2} minus $k$, where $k$ comes from the fact that the Poincar\'e dual of $b$ is a codimension $1$ cycle and there are $k$ boundary points at which we are evaluating. As a result, we obtain \eqref{eq:virmbbb}.

Since $\eL'$ and $\eL''$ are Lagrangian torus fibers of $M$, $\mu(\beta_{M})$ is twice the intersection number between $\beta_{M}$ and the toric boundary divisor of $M$;
in particular, $\mu(\beta_{M}) \ge 0$ by positivity of intersections, and $\mu(\beta_{M})$ is  even.  We also know by \eqref{eq:virmbbb} that $2-\mu(\beta_{M}) \ge 0$, because we have assumed that 
$((S,\ul{z},\ul{z}^+,\ul{m}),u,\xi)$ contributes to $m_{k,\beta}^{\bfb,b_0}(b_+,\dots,b_+)$
and there are no negative degree cochains on $L$.
Together with the assumption that $\mu(\beta_{M}) \neq 0$, we deduce that $\mu(\beta_{M})=2$
and $m_{k,\beta}^{\bfb,b_0}(b_+,\dots,b_+)$ is necessarily of cohomological degree $0$.
That means that $m_{k,\beta}^{\bfb,b_0}(b_+,\dots,b_+)$ is a multiple of the unit $e_L$.
Summing over all possible $\beta$,  the result follows.
\end{proof}

Therefore, the hypotheses of Theorem \ref{t:superpotential2} will hold whenever Lemma \ref{l:H1embed} is applicable.

\subsection{Tropical picture\label{Sec:tropical}}

To compute the superpotential $W^{\bfb}$ 
of $L$ for $\bfb$ as in Lemma \ref{l:H1embed}, it is sufficient to classify Maslov index two $J_{M}$-holomorphic curves with boundary on $\eL' \sqcup \eL''$. 
Since $\eL'$ and $\eL''$ are product Lagrangians and $J_{M}$ is a product complex structure,
we can study $J_{M}$-holomorphic curves with boundary on $\eL' \sqcup \eL''$ by projecting to the two $\mathbb{P}^1$
factors of $M$.
In other words, we want to classify holomorphic maps $v=(v_1,v_2)$ for $v_1,v_2:\Sigma \to \mathbb{P}^1$ with boundary on the 
respective projections of $\eL'$ and $\eL''$ such that the sum of the Maslov indices of $v_1$
and $v_2$ is $2$.
This classification is undertaken  in Section \ref{ss:classification}, but in preparation, we 
find it helpful to give a tropical picture which motivates the result.

Let $\Log:(\CC^*)^2 \to \bR^2$ be $\Log(z_1,z_2)=(\log|z_1|,\log|z_2|)$.
Let $p'=\Log(\eL')$ and $p''=\Log(\eL'')$ be two points in $\bR^2$.
The expected paradigm is (see \cite{Gross10, GPS10}):

\begin{philosophy}
The tropicalization of a connected $J_{M}$-holomorphic curve  with boundary on $\eL' \sqcup \eL''$
should give a  `broken' tropical curve $\gamma$ with boundary on $p' \cup p''$. That is, 
$\gamma$ is the image of a continuous map $h$ from a connected weighted finite graph without bivalent vertices $\Gamma$ to $\bR^2$ such that: 
\begin{enumerate}
\item if $v$ is a vertex of $\Gamma$ such that
$h(v)=p'$ (or $p''$), then $v$ is univalent; 
\item for every edge $e$ of $\Gamma$, $h|_e$ is an embedding and $h(e)$ is a line segment of rational slope; 
\item at every vertex $v$ of $\Gamma$ such that $h(v)\not \in \{p',p''\}$,  the balancing condition\footnote{i.e. $\sum_{e \text{ adjacent to } v} w_es_e=0 $, where $w_e$ is the weight of $e$ and $s_e$ is the primitive direction of $h(e)$ pointing towards $h(v)$.} holds;
\item if $h(e)$ has infinite length (i.e. is an unbounded edge), then the primitive direction of $h(e)$ belongs to $\{(\pm 1,0), (0, \pm 1)\}$.
\end{enumerate}
\end{philosophy}

Vertices and edges of $\gamma$ are defined to be the images of the vertices and edges of $\Gamma$.
The Maslov index of $\gamma$ is defined to be twice the number of unbounded edges of $\gamma$ (counted with multiplicity).
The genus of $\gamma$ is defined to be the rank of $H_1(\Gamma)$.

\begin{example}\label{e:Yshape}
Suppose the $x$-coordinate of $p'$ is smaller than that of $p''$.
Then there is a genus $0$ broken tropical curve $\gamma$ with $3$ edges $e_0,e_1,e_2$ such that
$e_0$ is adjacent to $p'$ with primitive direction $(1,-1)$, $e_1$ is adjacent to $p''$ with primitive direction $(-1,-1)$
and $e_2$ is a multiplicity $2$ unbounded edge whose primitive direction is $(0,1)$ (see Figure \ref{fig:Yshape}).  
\end{example}

\begin{figure}[ht]
\begin{center}
 \includegraphics{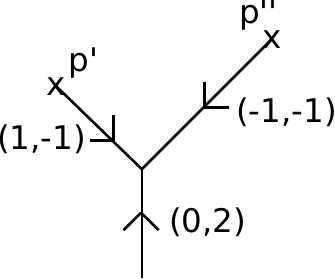} \end{center}
 \caption{A broken tropical curve with boundary on $p' \cup p''$} \label{fig:Yshape}
 
\end{figure}

In our situation, due to Lemma \ref{l:bCondition} (or more specifically \eqref{eq:bCondition}), we are interested in two cases:
\begin{enumerate}
 \item unions of two broken tropical curves $\gamma'$ and $\gamma''$, such that $p'$ is a vertex of $\gamma'$ but not $\gamma''$;
 and $p''$ is a vertex of $\gamma''$ but not $\gamma'$;
 \item broken tropical curves $\gamma$  such that both $p'$ and $p''$ are vertices of $\gamma$.
\end{enumerate}

It will be useful to impose the following `tropical general position' assumption. 

\begin{assumption}\label{a:general}
 The slope of the straight line joining $p'$ and $p''$ is irrational (here $\infty$ is regarded as rational). 
\end{assumption}

\begin{lem}\label{l:generalPosition}
If Assumption \ref{a:general} is satisfied, then there is no non-constant 
 Maslov index zero broken tropical curve with boundary on $p' \cup p''$.
\end{lem}

\begin{proof}
By definition, a Maslov index zero  broken tropical curve admits no unbounded edges. 
If such a tropical curve is not a constant, the balancing condition shows that it must be a straight line with rational slope joining 
 $p'$ and $p''$, which does not exist by assumption.
\end{proof}


\begin{rmk}\label{r:MaslovAdd}
 We define the Maslov index of a union of broken tropical curves to be the sum of the Maslov indices of the components.
 Therefore, by Lemma \ref{l:generalPosition}, if the union has Maslov index $2$
 then it is composed of exactly one Maslov $2$ broken tropical curve and some number of constant tropical curves.
\end{rmk}

We now study the possible 
 Maslov index two broken tropical curves  with boundary on $p' \cup p''$.

\begin{lem}\label{l:preClassification}
Let $\gamma$ be a Maslov $2$ broken tropical curve with boundary on $p' \cup p''$.
\begin{enumerate}
\item If $p'' \notin \gamma$ (resp.  $p' \notin \gamma$), then
$\gamma$ consists of one unbounded edge of multiplicity one emanating from $p'$ (resp. $p''$); this edge has primitive direction $(1,0)$, $(-1,0)$, $(0,1)$ or $(0,-1)$.

\item If $p' \cup p'' \in \gamma$ and $e_0,e_1,e_2$ are edges such that  $e_0$ is adjacent to $p'$, $e_1$ is adjacent to $p''$ and
$e_2$ is the unbounded edge (necessarily) with multiplicity one, then the sum of the weighted directions of $e_0,e_1,e_2$ is $0$.
\end{enumerate}

\end{lem}

\begin{proof}
Since $\gamma$ has Maslov index $2$, it has exactly one unbounded edge, necessarily of multiplicity $1$. Statement $(1)$ follows.

If $p' \cup p'' \in \gamma$, then clearly the edges $e_0,e_1,e_2$ are distinct.
The sum of the weighted directions of $e_0,e_1,e_2$ 
is the sum of the balancing conditions at all vertices of $\gamma$ other than $p'$ and $p''$.
Therefore, it must vanish.
\end{proof}

We need to further analyse case $(2)$ in Lemma \ref{l:preClassification}. 
Suppose from now on that $p'$ is in the third quadrant and $p''$ is in the first quadrant of $\bR^2$.

\begin{lem}\label{l:cyl1}
Suppose that Assumption \ref{a:general} is satisfied.
Let $\gamma$ be a broken tropical curve as in case $(2)$ of Lemma \ref{l:preClassification}.
If $e_2$ has direction $(1,0)$, $e_0$ has weighted direction $(p,q)$ and $e_1$ has weighted direction $(-(p+1),-q)$,
then we have
\begin{align}
\frac{q}{p+1}<m<\frac{q}{p} \label{eq:ineq}
\end{align}
where $m$ is the slope between the line joining $p'$ and $p''$.

Conversely, for every pair of integers $(p,q)$  such that \eqref{eq:ineq} is satisfied, there is a unique genus $0$ broken tropical curve with boundary on $p',p''$, 
consisting of the edges $e_0,e_1,e_2$  with weighted directions $(p,q),(-(p+1),-q)$ and $(1,0)$, respectively.
\end{lem}

\begin{proof}
We leave the proof as an exercise. See Figure \ref{fig:FromLeft}.
\end{proof}

\begin{figure}[ht]\begin{center}
 \includegraphics{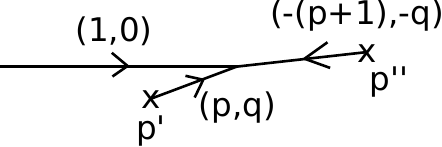}\end{center}
 \caption{A genus $0$ broken tropical curve as in case $(2)$ of Lemma \ref{l:preClassification}} \label{fig:FromLeft}
 
\end{figure}

Similarly, we have

\begin{lem}\label{l:cyl2}
Suppose that Assumption \ref{a:general} is satisfied.
Let $\gamma$ be a broken tropical curve as in case $(2)$ of Lemma \ref{l:preClassification}.
If $e_2$ has direction $(0,1)$, $e_0$ has weighted direction $(p,q)$ and $e_1$ has weighted direction $(-p,-(q+1))$,
then we have
\begin{align}
\frac{q}{p}<m<\frac{q+1}{p} \label{eq:ineq2}
\end{align}
where $m$ is the slope between the line joining $p'$ and $p''$.

Conversely, for every pair of integers $(p,q)$  such that \eqref{eq:ineq2} is satisfied, there is a unique genus $0$ broken tropical curve with boundary on $p',p''$, 
consisting of the edges $e_0,e_1,e_2$  with weighted directions $(p,q),(-p,-(q+1))$ and $(0,1)$, respectively.

\end{lem}

\begin{example}\label{e:annulus}
 When $0<m<1$, the tropical curve in Lemma \ref{l:cyl2} with the smallest $p$ has $p=1$ and $q=0$ (see Figure \ref{fig:Annuli} with $p_0$
 and $p_1$ understood as $p'$ and $p''$, respectively).
\end{example}

By symmetry, there are analogous results when $e_2$ has direction $(-1,0)$ or $(0,-1)$.

\begin{rmk}
There are higher genus Maslov $2$ broken tropical curves with boundary on $p' \cup p''$ (see Figure \ref{fig:Genus1}).
However, we will see that they only contribute higher order terms (in the adic filtration) to the bulk deformed superpotential, which will mean we do not need a classification of these curves to prove existence of critical points for the superpotential. 
\end{rmk}

\begin{figure}[ht] \begin{center}
 \includegraphics{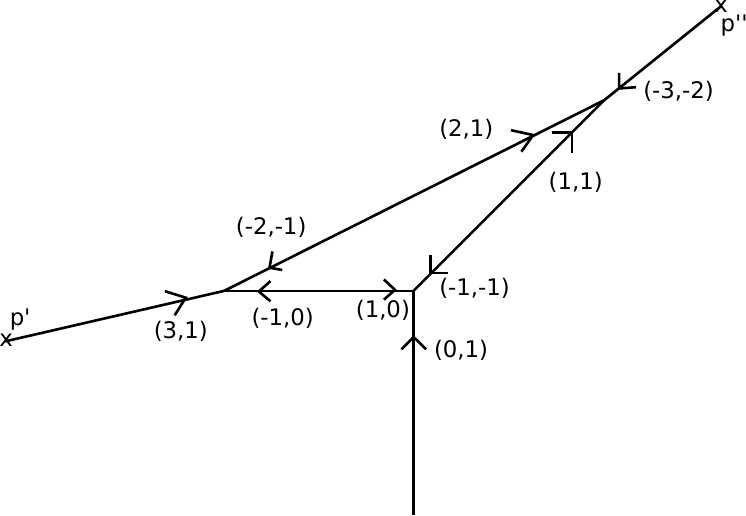}\end{center}
 \caption{A genus $1$ broken tropical curve as in case $(2)$ of Lemma \ref{l:preClassification}} \label{fig:Genus1}
 
\end{figure}

\subsection{Maslov two holomorphic curves}\label{ss:classification}

The tropical picture is heuristic, for two reasons. First, we have not justified that the $\Log_t$-images of a 
$J_{M}$-holomorphic curve with boundary on $\eL' \cup \eL''$ converge to a broken tropical curve.
More importantly, given a broken tropical curve $\gamma$, we have not proved that there is a $J_{M}$-holomorphic curve
with boundary on $\eL'$ and $\eL''$
whose tropicalization is $\gamma$.
In this section, we use  the tropical picture as a guide to help us locate and study holomorphic curves of Maslov index two.

We write $\eL'=\eL'_1 \times \eL_2'$ and $\eL''=\eL''_1 \times \eL_2''$.
We identify $M$ with $\mathbb{P}^1 \times \mathbb{P}^1$
in such a way that $\eL'=\{|z|=r_1'\} \times \{|w|=r_2'\}$
and $\eL''=\{|z|=r_1''\} \times \{|w|=r_2''\}$.
We assume that 
\[
r_1' < 1< r_1'' \qquad \textrm{and} \qquad r_2'< 1< r_2'', 
\] which correspond to the assumption that $p'$ and $p''$  are in the  
in the third respectively first quadrants. 
Let $C_i$ denote the cylinder bound by  $\eL'_i$ and $\eL''_i$, for $i=1,2$.
Moreover, we assume that 
\begin{align}
 \frac{r_1'}{r_1''} <\frac{r_2'}{r_2''} \label{eq:m<1}
\end{align}
which corresponds to the slope condition $0<m<1$, cf. Example \ref{e:annulus}.

Assumption \ref{a:general} translates to the following.
\begin{assumption}\label{a:irrational}
There are no integers $a,b \neq 0$ such that $(r_1' / r_1'' )^a = (r_2' / r_2'')^b$.
\end{assumption}

For the reader's convenience, we recall the following well-known fact.
\begin{lemma}
Let $\Sigma$ be a compact Riemann surface with boundary and let $\mathbb{P}_{\Delta}$ be a smooth toric variety.
Let
$v:\Sigma \to \mathbb{P}_{\Delta}$ be a holomorphic map such that $v(\partial \Sigma)$
is contained in a finite union of Lagrangian torus fibers.
Then the Maslov index of $v$ is given by $\mu(v)=2(v_*[\Sigma]) \cdot D$, where $D$ is the toric boundary divisor. In particular, if $\mathbb{P}_{\Delta}=\mathbb{P}^1$ and $v(\partial \Sigma)$ is contained in a finite union of circles $\{|z|=r_i\}$ for some $r_i >0$, then $\mu(v)=2(v_*[\Sigma]) \cdot [\{0,\infty\}]$.
\end{lemma}

The following lemma explains the importance of Assumption \ref{a:irrational}.

\begin{lem}[cf. Lemma \ref{l:generalPosition}]\label{l:Maslov0}
Suppose $v:\Sigma \to M$ is a $J_M$-holomorphic map satisfying \eqref{eq:bCondition}.
If Assumption \ref{a:irrational} holds, then $\mu(v) \neq 0$.
\end{lem}

\begin{proof}

Suppose not.
Let $v=(v_1,v_2):\Sigma \to M$ be  a $J_{M}$-holomorphic curve satisfying \eqref{eq:bCondition}
such that $\mu(v)=0$.
Since $\mu(v)=0$ and $J_M$ is split, we have $\mu(v_1)=\mu(v_2)=0$ by positivity of intersections.
If $\Sigma$ has two connected components $\Sigma_0$ and $\Sigma_1$, then at least one of $v|_{\Sigma_i}$, say $v|_{\Sigma_0}$, is not a constant.
Since $\Sigma_0$ has only one boundary component, then either $v_1|_{\Sigma_0}$ or $v_2|_{\Sigma_0}$ surjects to a disc and has $\mu \ge 2$.
This gives a contradiction.
Therefore, $\Sigma$ is connected.

By the boundary conditions, neither $v_1$ nor $v_2$ can be a constant map.
Therefore, we must have  that $v_1$ and $v_2$ surject onto the cylinders $C_1$ and $C_2$, respectively.
This is a contradiction because
it is well-known that no $\Sigma$ can simultaneously surject onto two annuli 
$\{r_1' \le |z| \le r_1''\}$ and $\{r_2' \le |w| \le r_2''\}$
which satisfy Assumption \ref{a:irrational}.

We present a proof here for the sake of completeness.

Let the two boundary components of $\Sigma$ be $\partial_0$ and $\partial_1$.
To simplify notation, let the two annuli be
$\{1 \le |z| \le R_1:=\frac{r_1''}{r_1'}\}$
and
$\{1 \le |z| \le R_2:=\frac{r_2''}{r_2'}\}$.
Without loss of generality, we can assume that $|v_i|_{\partial_0}|=1$ for both $i=1,2$.
Let 
\begin{align}
g(z)=\log|v_1(z)|^2-\alpha \log|v_2(z)|^2=2(\log|v_1(z)|-\alpha \log|v_2(z)|)
\end{align} 
where $\alpha=\log(R_2)/\log(R_1)$.
Clearly, $g$ is a harmonic function which vanishes on $\partial \Sigma$.
Therefore $g$ vanishes everywhere.
Note that $g$ is the real part of the (multi-valued) holomorphic function 
\begin{align}
h(z)=\log(v_1(z)^2)-\alpha \log(v_2(z)^2) \label{eq:h}
\end{align}
Therefore, $e^h$ is a (single-valued) constant function. By taking derivative, we have $e^h h'(z)=0$ and hence $h'(z)=0$.

Taking the derivative of \eqref{eq:h}, we get
\begin{align}
\frac{v_1'(z)}{v_1(z)}=\alpha \frac{v_2'(z)}{v_2(z)}
\end{align}
By integrating that identity over a closed curve parallel to $\partial_0$ and applying the residue theorem, we see that $\alpha$ is rational.
This contradicts Assumption \ref{a:irrational}.
\end{proof}



\begin{lem}[cf. Lemma \ref{l:preClassification} and Remark \ref{r:MaslovAdd}]\label{l:preCl}
Suppose that Assumption \ref{a:irrational} holds, 
and that $v=(v_1,v_2):\Sigma \to M$ is a Maslov index two $J_M$-holomorphic map satisfying \eqref{eq:bCondition}.
\begin{enumerate}
\item If $\Sigma$  is disconnected with connected components $\Sigma_0$ and $\Sigma_1$, then one of $v|_{\Sigma_i}$, say $v|_{\Sigma_1}$, is a constant. Then for $\Sigma_0$, one  $v_i|_{\Sigma_0}$ is a degree one map to a disc and the other is a constant.

\item If $\Sigma$ is connected, then either $v_1$ or $v_2$ is a (possibly unramified) branched  covering of a cylinder, 
and the other surjects to a disc.
\end{enumerate}

\end{lem}

\begin{proof}
If $\Sigma$  has two connected components $\Sigma_0$ and $\Sigma_1$, then $\partial \Sigma_i$ is connected.
Therefore, if $v_i|_{\Sigma_j}$ is not a constant, then $\mu(v_i) \ge 2$.
That means that $3$ of the $4$ maps  $\{v_i|_{\Sigma_j}\}_{i,j}$ have Maslov zero and are hence constant (by Lemma \ref{l:Maslov0}).
Suppose $v_i|_{\Sigma_0}$ is not a constant; then it is either a branched covering of a disc or it surjects onto $\mathbb{P}^1$.
However, if it surjects to $\mathbb{P}^1$, then it must have $\mu(v_i) \ge 4$.
Since $\mu(v_i) =2$, we know that $v_i|_{\Sigma_0}$ has degree $1$ and surjects to a disc.

Now asssume that $\Sigma$  is connected with $2$ boundary components.
Without loss of generality, we assume that $\mu(v_1)=0$ and $\mu(v_2)=2$.
Since $\mu(v_1)=0$, it has to surject onto the cylinder bounded by $\eL'_1$ and $\eL''_1$.
On the other hand, $\mu(v_2)=2$ implies that the image of $v_2$ is not the entire $\mathbb{P}^1$
and not the cylinder bounded by $\eL'_2$ and $\eL''_2$.
Therefore, $v_2$ surjects to a disc.
\end{proof}

Let $v$ be a holomorphic curve as in case $(2)$ of Lemma \ref{l:preCl}, 
such that $v_2$ surjects to a disc $D$.
Let $D^\circ$ be the interior of $D$.
By the Lagrangian boundary conditions for $v$, there are two possibilities:
\begin{enumerate}
 \item $D^\circ$ is the component of the complement of $\eL_2''$ that contains $\eL_2'$;
 \item $D^\circ$ is the component of the complement of $\eL_2'$ that contains $\eL_2''$.
\end{enumerate}
By the obvious symmetry, it is sufficient to analyse the first situation.
In this case, we have
\begin{align}
(v_1)_*[\Sigma,\partial \Sigma]=p [C_1] \quad \text{ and } \quad (v_2)_*[\Sigma,\partial \Sigma]=[D]+q [C_2] \label{eq:homo}
\end{align}
for some integers $p \ge 1$ and $q \ge 0$, where the relative classes are in $H_2(S^2_{2B+C},\eL'_1 \cup \eL_1'')$
and $H_2(S^2_{2a},\eL'_2 \cup \eL_2'')$, respectively.
Let $A_{p,q} \in H_2(M,\eL' \cup \eL'')$ be the class characterized by the property that
the projection to the two $\bP^1$ factors are $p [C_1]$ and $[D]+q [C_1]$, respectively, so
\begin{align}
 v_*[\Sigma,\partial \Sigma]=A_{p,q} \label{eq:Apq}
\end{align}
for some integers $p \ge 1$ and $q \ge 0$.

We can give a classification of such $v$ when $(p,q)=(1,0)$ (cf. Example \ref{e:annulus}).

\begin{lem}\label{l:algebraic count}
 There exists a $J_{M}$-holomorphic curve $v:\Sigma \to M$ with boundary on $\eL' \cup \eL''$ in class $A_{1,0}$.
 Moreover, for a generic pair of points $q' \in \eL'$ and $q'' \in \eL''$, 
 the algebraic count of unparametrized holomorphic curves in $A_{1,0}$ such that $q',q'' \in v(\partial \Sigma)$ is $\pm 1$.
 \end{lem}

\begin{proof}

First, we prove the existence of $v$.
Let $\Sigma$ be the annulus $C_1$
and define $v_1:\Sigma \to C_1 \subset \mathbb{P}^1$ to be the inclusion map.

By rescaling, we identify $D$ with the unit disc,  $\eL''_2$ with the unit circle
and $\eL'_2$ with the circle of radius $r_0:=\frac{r_2'}{r_2''}$.
Let $l \subsetneqq \eL'_2$ be a closed arc (usually called a `slit' in the literature, cf. \cite[Chapter $6$]{Ahlfors}, \cite[Chapter $7$]{Nehari})
and define $D_l:=D^\circ \setminus l$.
In Lemmas \ref{l:slitdomain} and \ref{l:samelength}, we provide proofs of the following classical facts:
\begin{enumerate}
\item For every $0<r<r_0$, there is a biholomorphism from  $\{r<|z|<1\}$ to  $D_l$ for some $l \subset \eL'_2$.
Moreover, the biholomorphism can be smoothly extended up to its closure.
\item Given two slits $l_1,l_2 \subset \eL_2'$, the domain $D_{l_1}$
is conformally isomorphic to $D_{l_2}$ if any only if $l_1$ and $l_2$ are of the same length.
\end{enumerate}
By assumption, $\Sigma=\{\frac{r_1'}{r_1''} \le |z| \le 1\}$ and $\frac{r_1'}{r_1''}<\frac{r_2'}{r_2''}=r_0$ (see \eqref{eq:m<1}).
Therefore, we have a holomorphic map 
$v_2:\Sigma \to S^2_{2a}$ such that the two boundary components are mapped to $\eL''_2$ and $l \subset \eL'_2$
for some $l$, respectively.
This proves the existence.

Conversely, 
let $v=(v_1,v_2)$ be a holomorphic curve in class $A_{1,0}$.
Since $v_1$ is a degree one map with the boundary components of $\Sigma$ going to $\eL_1'$ and $\eL_1''$, respectively, 
$v_1$ must be a biholomorphism.
Therefore, we can identify $\Sigma$ with $C_1$ via $v_1$.

On the other hand, $v_2$ is a degree $1$ map
onto the disc $D$.
Therefore $v_2|_{\Sigma \setminus \partial \Sigma}$ is a biholomorphism, so $v_2(\Sigma \setminus \partial \Sigma)$ 
is $D_l$ for some $l$.

Now, let $q'=(q_1',q_2') \in \eL'$ and $q''=(q_1'',q_2'')\in \eL''$.
By identifying $\Sigma$ with the annulus bound between $\eL'_1$ and $\eL''_1$, we can suppose that 
$q_1', q_1'' \in \partial \Sigma$.
To prove the last statement, it suffices to show that the algebraic count of $v_2$ such that $v_2(q_1')=q_2'$
and $v_2(q_1'')=q_2''$ is $\pm 1$.

Let $v_2:\Sigma \to D$ be a degree $1$ holomorphic map with boundary on $\eL'_2$ and $\eL''_2$ as above.
By the classification, all other degree $1$ holomorphic maps with boundary on $\eL'_2$ and $\eL''_2$ are given by
\begin{align}
 v_{2,\theta_1,\theta_2}(z)= e^{i \theta_2} v_2(e^{i \theta_1} z) \label{eq:allmaps}
\end{align}
for some $\theta_1,\theta_2 \in [0,2\pi]$.
Moreover, by automatic regularity, $v_{2,\theta_1,\theta_2}$ is regular for all $\theta_1$ and $\theta_2$.
Therefore, to show that the algebraic count is $\pm 1$, it suffices to show that
\begin{align}
&S^1 \times S^1 \to \eL'_2 \times \eL''_2 \\
&(\theta_1,\theta_2) \mapsto (\ev_1(\theta_1,\theta_2), \ev_2(\theta_1,\theta_2)):=(v_{2,\theta_1,\theta_2}(q_1'), v_{2,\theta_1,\theta_2}(q_1'')) 
\end{align}
is a degree $1$ map.

Let  $f:S^1 \to S^1$ be the smooth function given by $f(\theta_1)=\arg(q_2'')-\arg (v_2(e^{i\theta_1} q_1''))$.
By \eqref{eq:allmaps}, it is clear that for each $\theta_1$, there is a unique $\theta_2$
such that  $\ev_2(\theta_1,\theta_2)=q_2''$, indeed that $\theta_2$ is given by $f(\theta_1)$.
Moreover, since $v_2|_{\eL''_1}: \eL_1'' \to \eL_2''$ is a diffeomorphism, 
$v_2(e^{i\theta_1} q_1'')$ and hence $f$ has degree $\pm 1$.
The problem now reduces to showing that 
\begin{align}
&S^1 \to \eL'_2  \\
&\theta_1 \mapsto \ev_1(\theta_1,f(\theta_1))=v_{2,\theta_1,f(\theta_1)}(q_1') \label{eq:reduce}
\end{align}
is a degree $1$ map.
Since $v_2(\partial \Sigma) \cap \eL'_2$ is contractible, the 
map \eqref{eq:reduce} is homotopic to $f$, which has degree (plus or minus) one.
Therefore, the
algebraic count of  $v_2$ such that $v_2(q_1')=q_2'$
and $v_2(q_1'')=q_2''$ is indeed $\pm 1$.
This completes the proof.
\end{proof}

We now address the two classical facts used in the proof of Lemma \ref{l:algebraic count}.
Recall from the second paragraph of the proof of Lemma \ref{l:algebraic count} that 
$r_0:=\frac{r_2''}{r_2'}$, which is the radius of $\eL_2'$ when we identify $\eL_2''$ with the unit circle.

\begin{lemma}\label{l:slitdomain}
Let $r_1 \in (0,r_0)$. There is a biholomorphism from  $A:=\{r_1<|z|<1\}$ to  $D_l$ for some $l \subset \eL'_2$.
Moreover, the biholomorphism can be smoothly extended up to its closure.
\end{lemma}

\begin{proof}
Let  $a \in A$ and $G(z,a)$ be the Green's function, i.e. the unique function determined by the conditions: 
\begin{enumerate}
\item 
$G(z,a)|_{\partial \overline{A}}=0$ 
\item  $G(z,a)+\log|z-a|$ is harmonic and smooth everywhere on $A$.
\end{enumerate}
We denote the inner and outer boundary components of $\partial \overline{A}$ by $\partial_0$ and $\partial_1$.
Let $w_i$ for $i=0,1$ be the corresponding harmonic measures.
That is, $w_i$ is the unique smooth harmonic function on $\overline{A}$
such that $w_i|_{\partial_i}=1$ and $w_i|_{\partial \overline{A} \setminus \partial_i}=0$.
More explicitly, $w_0(z)=\frac{\log(|z|)}{\log(r_1)}$ and $w_1=1-w_0$.


Since $\Delta G(z,a)=-2\pi\delta_a$,
for any harmonic function $w:A\to \mathbb{R}$, we have the identity (see e.g. \cite[Chapter $1$]{Nehari})
\begin{align}
w(a)=\frac{-1}{2\pi}\int_{\partial \overline{A}} w(z) \frac{\partial G(z,a)}{\partial n}
\end{align}
where $\frac{\partial}{\partial n}$ refers to the outward normal derivative.
When $w=1$, it gives
\begin{align}
\int_{\partial_0} \frac{\partial G(z,a)}{\partial n}+\int_{\partial_1} \frac{\partial G(z,a)}{\partial n}=-2\pi \label{eq:period}
\end{align}
When $w=w_i$ for $i=0,1$, it gives
\begin{align}
w_i(a)=\frac{-1}{2\pi}\int_{\partial \overline{A}} w_i(z) \frac{\partial G(z,a)}{\partial n}=\frac{-1}{2\pi}\int_{\partial_i} \frac{\partial G(z,a)}{\partial n}  \label{eq:normal}
\end{align}
Therefore, we have $\beta:=-\int_{\partial_1} \frac{\partial G(z,a)}{\partial n}=2\pi w_1(a) \in (0,2\pi)$.

On the other hand, a direct calculation gives
\begin{align}
\int_{\partial_0} \frac{\partial w_0(z)}{\partial n}=-\frac{2\pi}{\log(r_1)}=-\int_{\partial_1} \frac{\partial w_0(z)}{\partial n}.
\end{align}
For $c=\frac{(2\pi-\beta)\log(r_1)}{2\pi}$, we have
\begin{align}
\int_{\partial_1} \frac{\partial (-G(z,a)+c w_0)}{\partial n}=\beta+c\frac{2\pi}{\log(r_1)}=2\pi; \\
\int_{\partial_0} \frac{\partial (-G(z,a)+c w_0)}{\partial n}=2\pi-\beta-c\frac{2\pi}{\log(r_1)}=0.
\end{align}
That means that the harmonic conjugate $H(z)$ of $-G(z,a)+c w_0$ is a multi-valued function with period
$2\pi$ on $\partial_1$ and period $0$ on $\partial_0$.

The map $F(z)=e^{-G(z,a)+c w_0+iH(z)}$ is therefore
a single valued holomorphic function which maps $a$ to the origin, $\partial_0$ to $\{|z|=e^{c}\}$ with degree $0$
and  $\partial_1$ to $\{|z|=1\}$ with degree $1$.
By a routine argument (see e.g. \cite[Theorem $10$ of Section $5$ of Chapter $6$]{Ahlfors}), one can check that $F$ is a biholomorphism from $A$ to the slit domain $D \setminus \{F(\partial_0)\}$.

Since $0<\beta<2\pi$, we have $\log(r_1)<c<0$. 
On the other hand, for any value $\beta_0$
between $0$ and $2\pi$, there is a unique $a$ (up to automorphism of $A$)
such that $\beta=2\pi w_1(a)=\beta_0$.
Since $r_1<r_0$, we can pick the $a$ such that the corresponding $c$ is $\log(r_0)$, so $F(\partial_0) \subset \eL'_2$.

Finally, since the boundary components of $\overline{A}$ are smooth analytic curves, we claim that  $F$ can be extended smoothly up to the boundary. More precisely, for each point $z_0 \in \partial \overline{A}$, the Riemann mapping theorem implies that we can find a small neighborhood $U \subset \overline{A}$
of $z_0$ such that $U$ is biholomorphic to the upper half of the unit disc by a biholomorphism 
sending $U \cap \overline{A}$ and $z_0$ to the interval $[-1,1]$ and the origin, respectively. 
We can assume $U$ is sufficiently small that $\log F(z)$ is single-valued and its real part $-G(z,a)+c w_0$ tends to a real constant as $z$ approaches 
$U \cap \overline{A}$.
Therefore, by the reflection principle, $\log F(z)$ and hence $F(z)$ can be extended over $z_0$.
\end{proof}

\begin{lemma}\label{l:samelength}
Given two slits $l_1,l_2 \subset \eL_2'$, the domain $D_{l_1}$
is conformally isomorphic to $D_{l_2}$ if and only if $l_1$ and $l_2$ are of the same length.
\end{lemma}

\begin{proof}
By the classification of multi-connected domains,  $D_{l_1}$ is biholomorphic to $A:=\{r_1<|z|<1\}$ for some $r_1$.
Let $F_1:A \to D_{l_1}$ be a biholomorphism such that 
the smooth extension of $F_1$ maps $\{|z|=1\}$ to the unit circle. 
Let $a_1\in A$ be the point such that $F_1(a_1)=0$.
Then $\log|F_1(z)|$ is a harmonic function which is smooth everywhere except having a log pole at $a_1$.
Moreover, it maps $\{|z|=1\}$ to $0$ and $\{|z|=r_1\}$ to $\log(r_0)$.
These conditions uniquely characterise $\log|F_1(z)|$, and so we have
\begin{align}
\log|F_1(z)|=-G(z,a_1)+\log(r_0) w_0 \label{eq:F1}
\end{align}
where $G(z,a_1)$ and $w_0$ are the Green's function and harmonic measure introduced in the proof of Lemma \ref{l:slitdomain}.

It follows that if  $D_{l_2}$ is biholomorphic to  $D_{l_1}$ and hence to $A$, then the 
biholomorpism $F_2:A \to D_{l_2}$ satisfies
\begin{align}
\log|F_2(z)|=-G(z,a_2)+\log(r_0) w_0 \label{eq:F2}
\end{align}
for the $a_2 \in A$ such that $F_2(a_2)=0$.

Moreover, we know the periods of the harmonic conjugates of \eqref{eq:F1} and \eqref{eq:F2} are the same, namely $2\pi$ on the outer boundary and $0$ on the inner boundary. In other words, we have
\begin{align}
\int_{\partial_1} \frac{\partial (-G(z,a_1)+\log(r_0) w_0)}{\partial n}=\int_{\partial_1} \frac{\partial (-G(z,a_1)+\log(r_0) w_0)}{\partial n}=2\pi; \\
\int_{\partial_0} \frac{\partial (-G(z,a_2)+\log(r_0)  w_0)}{\partial n}=\int_{\partial_0} \frac{\partial (-G(z,a_2)+\log(r_0)  w_0)}{\partial n}=0.
\end{align}
It implies that 
\begin{align}
\int_{\partial_1} \frac{\partial G(z,a_1)}{\partial n}=\int_{\partial_1} \frac{\partial G(z,a_1)}{\partial n}
\end{align}
By \eqref{eq:normal} and the equation afterwards, it means that
$w_1(a_1)=w_1(a_2)$.
Since $w_1(z)=1-\frac{\log(|z|)}{\log(r_1)}$, we have $w_1(a_1)=w_1(a_2)$ if and only if $|a_1|=|a_2|$.
Therefore, there is an automorphism $\phi$ of $A$ (a rotation) which sends $a_1$ to $a_2$.
As a result, the biholomorphism $F_2':=F_2 \circ \phi: A \to D_{l_2}$ satisfies
\begin{align}
\log|F_2'(z)|=-G(z,a_1)+\log(r_0) w_0
\end{align}
because it sends $a_1$ to the origin.
It follows that $F_1$ and $F_2'$ only differ  by a choice of harmonic conjugate
 of $-G(z,a_1)+\log(r_0) w_0$, and are hence related by $F_1=e^{i \alpha}F_2'$ for some $\alpha$. 
It follows that $l_1$ and $l_2$ have the same length.
\end{proof}

\begin{rmk}\label{r:cylLift}
For a circular cylinder $\{1/r \le |z| \le r\}$, any involution swapping the boundary components belongs to the family
$(z \mapsto \frac{e^{i\theta}}{z})_{\theta \in [0,2\pi]}$.
Consequently, there is a $1$-dimensional family of $2$ to $1$ branched coverings of the cylinder over  the unit disc.
In other words, the $1$ dimensional moduli space of cylinders can be identified with 
 the $1$ dimensional moduli space of double branched coverings of a disc with $2$ interior orbifold marked points (see Remark \ref{r:moduliexplain} and the paragraph before), which is in turn isomorphic to the $1$ dimensional 
moduli space of discs with $2$ interior marked points.

Let $v$ be as in Lemma \ref{l:algebraic count}, and let $z',z'' \in \partial \Sigma$ be such that $v(z')=q'$
and $v(z'')=q''$.
If we choose the involution of $\Sigma$ which exchanges $z'$ and $z''$,
the corresponding $2$ to $1$ branched covering of the unit disc together with $v$ will tautologically correspond
to a $J_{\eX}$-holomorphic orbifold disc $u:S \to \eX$ such that $[q',q''] \in u(\partial S)$.

Notice that, since the complex structure splits, we can write  $v=(v_1,v_2)$. It is easy to check that each $v_i$
is regular and hence $v$ is regular.
Alternatively, one can apply  automatic regularity in dimension $4$ (see \cite[Theorem 2]{HLSRegular}), which in this case says that $\mu(v)=2>1=4g+2\sigma-3$ and hence $v$ is regular.
(Here, $g=0$ is the genus of the domain and $\sigma=2$ is the number of boundary components of the domain.) 
By Lemma \ref{l:FredReg}, we infer that the corresponding holomorphic orbidisc $u$ is also regular.
In particular, $u$ will 
contribute to $m_{k}^{\bfb,b}$ whenever $\bfb_\orb \neq 0$.
Moreover,  
the unparametrized algebraic count in Lemma \ref{l:algebraic count} implies that
the virtual fundamental chain of the moduli containing $u$ (i.e. $\cM_{1,2, \bfx}(\Sym(\eL' \cup \eL'');u_*[S]; [\eX_1]^{\otimes 2},1)$) satisfies
\begin{align}
(\cM_{1,2, \bfx}(\Sym(\eL' \cup \eL'');u_*[S]; [\eX_1]^{\otimes 2},1), \ev_0)=\pm [\Sym(\eL' \cup \eL'')]
\end{align}
where $\bfx(j)=1$ for both $j=1,2$.
\end{rmk}

On the other hand, we can prove the non-existence of maps $v$ as in \eqref{eq:Apq} when $p=1$ and $q>0$.

\begin{lem}\label{l:Nonexistence}
There is no holomorphic curve with boundary on $\eL' \cup \eL''$ in class $A_{1,q}$ for $q>0$.
\end{lem}

\begin{proof}
Suppose that such a curve exists.
The hypothesis $p=1$ implies that $v_1$ is a biholomorphism to $C_1$.
Therefore, $\Sigma$ is a cylinder; we identify it with $\{  \frac{r_1'}{r_1''}  \le |z| \le 1\}$.
As before, we identify $D$ with the unit disc
and $L_2'$ with $\{|z|=r_0<1\}$.

By the definition of $A_{1,q}$,
the holomorphic map $v_2$ has degree $1$ to the disc $\{|z| \le r_0\}$
 and degree $q+1$ to the cylinder $\{r_0 \le |z| \le 1\}$.
Let $a \in \Sigma$ be the point such that $v_2(a)=0$.
The harmonic function $\log|v_2|$
is smooth everywhere except at $a$, where it has a log pole.
Moreover, it maps $\{|z|=1\}$ to $0$ and $\{|z|= \frac{r_1'}{r_1''}\}$ to  $\log(r_0)$.
Therefore, we have
\begin{align}
\log|v_2|=-G(z,a)+\log(r_0) w_0
\end{align}
where $G(z,a)$ and $w_0$ are the Green's function and harmonic measure respectively.

As in Lemma \ref{l:slitdomain}, we have
\begin{align}
\int_{\partial_1} \frac{\partial (-G(z,a)+\log(r_0) w_0)}{\partial n}=\beta+\log(r_0)\frac{2\pi}{\log(r_1'/r_1'')} \label{eq:nocurve1}
\end{align}
for some $\beta \in (0,2\pi)$.
By our assumption, we have $\frac{r_1'}{r_1''} <\frac{r_2'}{r_2''} =r_0<1$ so the quantity in \eqref{eq:nocurve1} is less than $4\pi$.

However, this is a contradiction. Indeed, the harmonic conjugate of $-G(z,a)+\log(r_0) w_0$, which is the harmonic conjugate of 
$\log|v_2|$, must have period $(q+1)2\pi \ge 4\pi$, because $v_2|_{\{|z|=1\}}$ has degree $q+1$.
This completes the proof.
\end{proof}

\begin{rmk}
 Lemma \ref{l:Nonexistence} corresponds to  the fact that in Lemma \ref{l:cyl2}, there is no broken tropical
 curve with $p=1$ and $q>0$ when $0<m<1$.
\end{rmk}

By symmetry, the analogues of Lemma \ref{l:algebraic count} and \ref{l:Nonexistence} hold when 
$v$ is a holomorphic curve as in case $(2)$ of Lemma \ref{l:preCl} (one for which $v_2$ surjects to the complementary component of $\eL_2'$ that contains $\eL_2''$).

\section{Concluding the proof}

So far, we have been using Assumption \ref{a:irrational}.
However, the fibres $\eL_0$ and $\eL_1$ in Theorem \ref{t:nondisLag} do not satisfy this assumption, 
because their  second factors coincide.
Therefore, we cannot apply the previous results directly to $\Sym(\eL)$.
We remedy this issue by applying what is commonly referred to as `Fukaya's trick'.

As in the previous section, let $\eL'$ and $\eL''$ be Lagrangian torus fibers in $M$ such that Assumption \ref{a:irrational} is satisfied.
Without loss of generality we may assume that $\eL'$ and $\eL''$
are close enough to $\eL_0$ and $\eL_1$ such that there is a \emph{diffeomorphism} 
$\Phi: M \to M$ satisfying
\begin{enumerate}
\item $\Phi(\eL)=\eL' \cup \eL''$, and
\item $\Phi^*J_{M}$ is $\omega_{M}$-tamed.
\end{enumerate}
Tautologically, the pull-back of $J_{M}$-holomorphic curves with boundary on $\eL' \cup \eL''$
under $\Phi$ are $\Phi^*J_{M}$-holomorphic curves with boundary on $\eL$, and vice versa.
Therefore, we can apply the previous discussion to study the superpotential of $\Sym(\eL)$
with respect to the symmetric product of (the integrable complex structure) $\Phi^*J_{M}$ on $\eX$.

\begin{rmk}
The superpotential will be different for different choices of $\eL'$ and $\eL''$.
That is because the corresponding almost complex structures $\Phi^*J_{M}$ differ, and there is  
wall-crossing when one interpolates between them (cf. \cite{Auroux-wall} and the scattering diagrams in \cite{Bousseau}). 
\end{rmk}

Now, we give a brief summary of the terms of the superpotential to which the $J_{M}$-holomorphic curves in 
Lemma \ref{l:preCl} and \ref{l:algebraic count} contribute.

Let $B^\circ_{1,1}, B^\circ_{1,2}, B^\circ_{2,1}$ and $B^\circ_{2,2}$ be the open discs depicted in Figure \ref{Fig:discs}, such that
\begin{enumerate}
 \item $B^\circ_{1,1}$ is the connected component of the complement of $\eL_1'$ not containing $\eL_1''$;
 \item $B^\circ_{1,2}$ is the connected component of the complement of $\eL_2'$ not containing $\eL_2''$;
 \item $B^\circ_{2,1}$ is the connected component of the complement of $\eL_1''$ not containing $\eL_1'$;
 \item $B^\circ_{2,2}$ is the connected component of the complement of $\eL_2''$ not containing $\eL_2'$.
\end{enumerate}

\begin{figure}[ht]
\begin{center} 
\begin{tikzpicture}[scale=0.6]

\draw[semithick] ellipse  (3.05 and 3.05);
\draw[semithick] (0.7,2.95) arc (10:-10:17);
\draw[semithick] (-0.68,2.96) arc (10:-10:17);
\draw[semithick, dashed] (-0.68,2.96) arc (170:190:17);
\draw[semithick, dashed] (0.7,2.95) arc (170:190:17);

\draw (-0.6,-3.6) node {$\EuScript{L}_1'$};
\draw (0.7,-3.6) node {$\EuScript{L}_1''$};
\draw (6.4,-.6) node {$\EuScript{L}_2'$};
\draw (6.4,0.6) node {$\EuScript{L}_2''$};

\draw[semithick]  (10,0) ellipse  (3.05 and 3.05);
\draw[semithick,dashed] (7.,-0.6) arc (100:80:17);
\draw[semithick] (7.,-.6) arc (260:280:17);
\draw[semithick,dashed] (7, 0.6) arc (100:80:17);
\draw[semithick] (7,.6) arc (260:280:17);

\draw (-2,0) node {$B_{1,1}^{\circ}$};
\draw (2,0) node {$B_{2,1}^{\circ}$};
\draw (10,-2) node {$B_{1,2}^{\circ}$};
\draw (10,+2) node {$B_{2,2}^{\circ}$};

\end{tikzpicture}
\end{center}
\caption{Labellings for disc classes\label{Fig:discs}}
\end{figure}
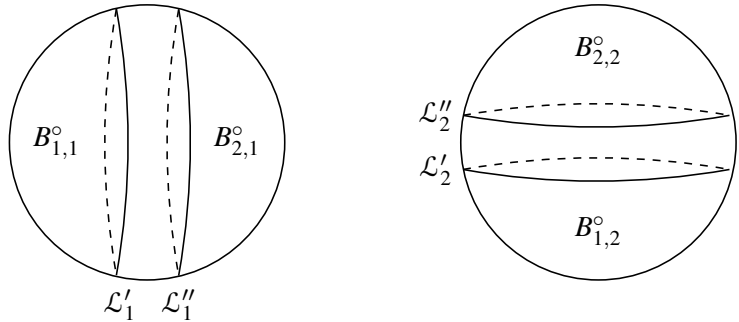

Let $B_{i,j}$ be the closure of $B^\circ_{i,j}$.
Let $\beta_{1,1}, \beta_{1,2}, \beta_{2,1}, \beta_{2,2}, \delta_1, \delta_2$ be the classes in  $H_2(M, \eL)$
such that their pushforward in $H_2(M, \eL' \cup \eL'')$ by $\Phi$
are the classes determined by
\begin{align}
 (\pi_1)_* \Phi_* \beta_{1,1}=B_{1,1} &\quad \text{ and } \quad (\pi_2)_* \Phi_* \beta_{1,1}=0 \\
 (\pi_1)_* \Phi_* \beta_{1,2}=0 &\quad \text{ and } \quad (\pi_2)_* \Phi_* \beta_{1,2}=B_{1,2} \\
 (\pi_1)_* \Phi_* \beta_{2,1}=B_{2,1} &\quad \text{ and } \quad (\pi_2)_* \Phi_* \beta_{2,1}=0 \\
 (\pi_1)_* \Phi_* \beta_{2,2}=0 &\quad \text{ and } \quad (\pi_2)_* \Phi_* \beta_{2,2}=B_{2,2} \\
 (\pi_1)_* \Phi_* \delta_1=C_1 &\quad \text{ and } \quad (\pi_2)_* \Phi_* \delta_1=0 \\
 (\pi_1)_* \Phi_* \delta_2=0 &\quad \text{ and } \quad (\pi_2)_* \Phi_* \delta_2=C_2
\end{align}
where $\pi_1:M \to S^2_{2B+C}$ and $\pi_2:M \to S^2_{2a}$ are the projection to the two factors.
In particular, the class $A_{p,q}$ in \eqref{eq:Apq} is given by $\Phi_* (p\delta_1+ \beta_{1,2}+(q+1) \delta_2)$.

\begin{figure}[ht]\begin{center}
 \includegraphics{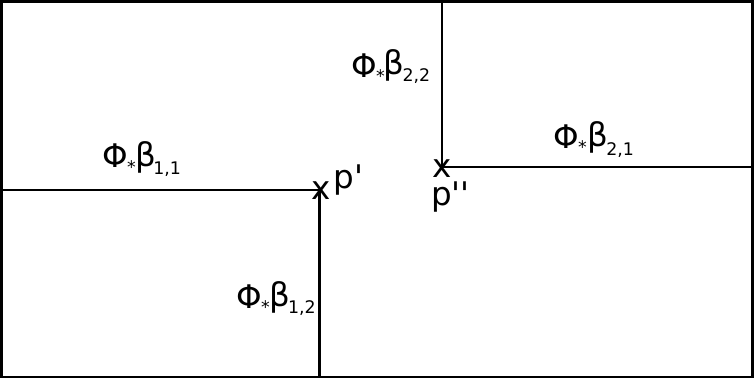}\end{center}
 \caption{A tropical picture of the four disc classes in the homology basis. The rectangle represents the moment map image of the toric boundary.}\label{fig:HomologyBasis}
\end{figure}

Note that, $H_2(M, \eL)$ is freely generated by these six classes.
Their symplectic areas are given by
\begin{align}
&\omega_{M}( \beta_{1,1})=\omega_{M}( \beta_{2,1})=B \\
&\omega_{M}( \beta_{1,2})=\omega_{M}( \beta_{2,2})=a \\
&\omega_{M}( \delta_1)=C, \quad \omega_{M}( \delta_2)=0
\end{align}

By Lemma \ref{l:preCl}, to compute the superpotential of $\Sym(\eL)$, there are two classes of  $\Phi^*J_{M}$-holomorphic maps $v:(\Sigma,\partial \Sigma) \to (M, \eL)$
we want to consider, namely, those such that $\Sigma$ is a disjoint union of two discs $\Sigma_0 \sqcup \Sigma_1$ and those such that $\Sigma$ is connected.
We consider the former case first. By Lemma \ref{l:preCl}, three of the maps $v_i|_{\Sigma_j}$ (for $i=1,2$, $j=0,1$) are constants mapping to a point on the respective Lagrangian (boundary conditions) and the remaining one is
a biholomorphism onto a disc.
Therefore for each $\beta_{k,l}$,  the Maslov index $2$ elements in $\cM_{1,0}(\Sigma, \eL,\Phi^*J_{M},\beta_{k,l})$
are precisely those $v$ in Lemma \ref{l:preCl}(1) such that the only non-constant map $v_i|_{\Sigma_j}$
is in class $\beta_{k,l}$.
Together with the trivial $2$ to $1$ covering from $\Sigma$ to the disc, these $v$ lift to a Maslov $2$ holomorphic disc $u$ in $\eX$
with boundary on $\Sym(\eL)$.
Moreover, by looking at the algebraic count of these $v$ as in Lemma \ref{l:algebraic count}, it is easy to see that the virtual fundamental chain of the moduli containing $u$ satisfies
\begin{align} 
(\cM_{1,0}(\Sym(\eL),\widetilde{\beta}_{k,l}, 1, 1), \ev_0)=\pm [\Sym(\eL)] \label{eq:VirChain1}
\end{align}
where 
$\widetilde{\beta}_{k,l} \in H_2(\eX, \Sym(\eL))$ is the class corresponding to $\beta_{k,l}$ (i.e. $u_*[S]=\widetilde{\beta}_{k,l}$ when $v_*[\Sigma]=\beta_{k,l}$).
Equation \eqref{eq:VirChain1} holds because all the elements in $(\cM_{1,0}(\Sym(\eL),\widetilde{\beta}_{k,l}, 1, 1)$ are regular (which can be proved directly or by applying automatic regularity to $v$, cf. Remark \ref{r:cylLift}) and it is easy to check that $\ev_0$ is a degree $\pm 1$ map to $\Sym(\eL)$ (see Remark \ref{r:orientation} below for the discussion of orientation).

We fix coordinates $(x_1,x_2,y_1,y_2)$ on $H^1(\Sym(\eL),\Lambda_0)/H^1(\Sym(\eL),2 \pi \sqrt{-1} \mathbb{Z})$ (playing the role of $Y_i$ in \eqref{eq:familiarW}) such that
the $4$ families of holomorphic discs in classes $ \widetilde{\beta}_{1,1},  \widetilde{\beta}_{1,2},  \widetilde{\beta}_{2,1}$ and $ \widetilde{\beta}_{2,2}$
contribute  the terms
 $T^Bx_1$, $T^ax_2^{-1}$, $T^By_1$  and $T^ay_2^{-1}$ in the superpotential, respectively (the areas of the discs and hence the $T$-powers are as stated by an application of Lemma \ref{l:Area}).

\begin{rmk}\label{r:orientation}
  The coordinates $(x_1,x_2,y_1,y_2)$ are obtained as follows.
  We have  identified $\eL'$ with $\{|z|=r_1'\} \times \{|w|=r_2'\} \subset \CC^2$.
  We apply the involution $w \mapsto 1/w$ to the second factor to identify 
  $\eL'$ with $\{|z|=r_1'\} \times \{|w|=1/r_2'\} \subset \CC^2$.  In this description, we give 
  each factor the counterclockwise orientation induced from $\CC$, and we use these orientations to define $x_1,x_2$.  We
 identify each factor of $\eL'$ with $\mathbb{R}/\mathbb{Z}$ by $\theta \mapsto e^{i\theta}$ and use the natural trivialization $T\mathbb{R}=\mathbb{R}^2$ to induce a trivialization of $T(\mathbb{R}/\mathbb{Z})$, which we call the translation-invariant trivialization.
  The trivialization of $T\eL'$ that we use to orient the moduli spaces of holomorphic discs is the product of the 
  translation-invariant trivializations in the factors of this chart $\CC^2$, compare to \cite[Section 8]{Cho-Clifford} (whose conventions we are following).
    Similarly, we apply $z \mapsto 1/z$ to the first factor to identify 
  $\eL''$ with $\{|z|=1/r_1''\} \times \{|w|=r_2''\} \subset \CC^2$ and take the induced orientation and translation-invariant trivialization on the factors.
  
  We now arrange that $(z,w) \mapsto (1/z,1/w)$ maps $\eL'$ to $\eL''$; this amounts to saying that the torus fibres $\eL'$ and $\eL''$ are swapped by a $\pi$-rotation of the moment rectangle;  this can be done without violating the `irrational slope' Assumption \eqref{a:irrational} because the `slope' varies continuously when we vary $r_1', r_2'$. 
By construction this map respects the orientations and trivializations chosen. 
\end{rmk}

There are $4$ more families of holomorphic discs contributing to $W$.
They are in classes $\beta_{2,1}+\delta_1$, $\beta_{2,2}+\delta_2$,  $\beta_{1,1}+\delta_1$ and $\beta_{1,2}+\delta_2$ (or more precisely, their corresponding classes in $H_2(\eX,\Sym(\eL))$), and the terms they contribute are  
$T^{B+C}x_1^{-1}$, $T^ax_2$,  $T^{B+C}y_1^{-1}$, and $T^ay_2$ respectively. 

From the discussion above, we conclude that the part of the superpotential of $\Sym(\eL)$ arising from the  contributions of smooth discs is given by
\begin{align}
 W_{\smooth}=T^a(x_2^{-1}+y_2^{-1}+x_2+y_2)+T^B(x_1+y_1)+T^{B+C}(x_1^{-1}+y_1^{-1}) \label{eq:smoothW}
\end{align}
Note that \eqref{eq:smoothW} precisely takes into account all the Maslov $2$ holomorphic curves of case $(1)$ of Lemma \ref{l:preCl}.

We now consider a bulk deformation $\mathbf{b}_{\orb} [\eX_1]$ for $\mathbf{b}_{\orb} \in \Lambda_+$.
At this point, the Maslov $2$ holomorphic curves from case $(2)$ of Lemma \ref{l:preCl} can also contribute to the bulk deformed superpotential.

Let $v:\Sigma \to M$ be a genus $g$ Maslov $2$ holomorphic curve as in Lemma \ref{l:preCl}, case $(2)$, and view $v$ as a $\Phi^*J_{M}$-holomorphic map
with boundary on $\eL$.
By Lemma \ref{l:preCl}, we have
\begin{align}
 v_*[\Sigma, \partial \Sigma]=\beta_{i,j}+\eta_1\delta_1+\eta_2\delta_2
\end{align}
for some $\eta_{1}, \eta_2  \ge 1$ and $i,j \in \{1,2\}$.
If there is a $2$ to $1$ branched covering map $\pi_{\Sigma}$ from $\Sigma$ to the unit disc $S$, then 
$\pi_{\Sigma}$ has $2g+2$ critical values (see Lemma \ref{l:RH}).
Therefore, the term of $W^\bfb$ of $\Sym(\eL)$ to which it contributes, via the tautological correspondence, is (cf. the end of Section \ref{s:orbifold})
\begin{align}
c\frac{\mathbf{b}_{\orb}^{2g+2}}{(2g+2)!} T^{\omega_{M}(\beta_{i,j})+\eta_1 C}(x_2y_2)^{\eta_2}(x_1y_1)^{-\eta_1}f(x,y) \label{eq:highG}
\end{align}
for some $c \in \mathbb{C}$, where $f(x,y)=x_1,x_2^{-1},y_1$ or $y_2^{-1}$ depending on $\beta_{i,j}$.

\begin{lemma}\label{l:lowCon}
 The lowest order contribution to $W^{\mathbf{b}_{\orb} [\eX_1]}$
 from Maslov $2$ holomorphic curves in Lemma \ref{l:preCl} case $(2)$ is given by
 \begin{align}
   \pm \frac{\mathbf{b}_{\orb}^2}{2} T^{a+C}(x_1y_1)^{-1}(x_2+y_2) \label{eq:lowOrder}
 \end{align}
\end{lemma}

\begin{proof}
Note that, the smallest possible exponent of $T$ in \eqref{eq:highG}
is achieved when 
$g=0$, $\beta_{i,j}=\beta_{1,2}$ or $\beta_{2,2}$, and $\eta_1=1$.
We proved in Lemma \ref{l:Nonexistence} that when $\eta_1=1$ and $\beta_{i,j} \in \{\beta_{1,2}, \beta_{2,2}\}$, there is no $v$ such that 
$\eta_2 > 1$.
When $\eta_2=1$, the holomorphic curves $v$ are classified in Lemma \ref{l:algebraic count}.
More precisely, they are regular and the algebraic count 
 in each of the classes $\beta_{1,2}+\delta_1+\delta_2$
and $\beta_{2,2}+\delta_1+\delta_2$ is $\pm 1$, and the two counts have the same sign by the symmetry of our choice of orientation and trivialization 
(see Remark \ref{r:orientation}, and also Remark \ref{r:sign}).

In other words, for $\beta^{lift} \in H_2(\eX,\Sym(\eL))$ being the corresponding lift of  $\beta_{1,2}+\delta_1+\delta_2$
or $\beta_{2,2}+\delta_1+\delta_2$, the elements in
$\cM_{1,2, \bfx}(\Sym(\eL);\beta^{lift}; [\eX_1]^{\otimes 2},1)$ are regular and
\begin{align}
(\cM_{1,2, \bfx}(\Sym(\eL);\beta^{lift}; [\eX_{1}]^{\otimes 2},1), \ev_0)=\pm [\Sym(\eL)] \label{eq:liftSign}
\end{align}
Moreover, the signs of \eqref{eq:liftSign} are the same for $\beta_{1,2}+\delta_1+\delta_2$
and $\beta_{2,2}+\delta_1+\delta_2$.

The contribution of these curves to $W^{\mathbf{b}_{\orb} [\eX_1]}$ is therefore given by \eqref{eq:lowOrder} (see Remark \ref{r:cylLift} and the end of Section \ref{s:orbifold}).
\end{proof}

From now on, we take $\mathbf{b}_{\orb}$ such that
\begin{align}
 \frac{\mathbf{b}_{\orb}^2}{2}=T^{B-a-C}
\end{align}
Then \eqref{eq:lowOrder} becomes $ \pm T^{B}(x_1y_1)^{-1}(x_2+y_2)$ which is of the same order as the term $T^B(x_1+y_1)$ in $W_{\smooth}$.

Write the toric boundary divisor of $M$ as $S_1+S_2$, where $S_1$ (resp. $S_2$) is the sum of the two spheres with smaller (resp. larger) area.
Let $D_{S_1}$ be the image of $S_1 \times M \subset M \times M$ in $\eX$ under the quotient map (i.e. this is the divisor of pairs of points at least one of which lies in $S_1$).
It defines a cycle $[D_{S_1}]$ in $H_6(\eX_0)=H_6(\eX)$.
Let $\bfb=\bfb_1 [D_{S_1}]+\mathbf{b}_{\orb} [\eX_1]$ with $\bfb_1 \in \Lambda_+$.
To conclude the proof of Theorem \ref{t:nondisLag2} (and hence Theorem \ref{t:nondisLag}), it suffices to find an appropriate $\bfb_1$
such that $W^\bfb$ has a critical point in $(\Lambda_0 \setminus \Lambda_+)^4$.

\begin{lemma}
There exists $\bfb_1 \in \Lambda_+$ such that
 $W^\bfb$ has $6$ critical points with $x_1=y_1$ and $x_2=y_2$.
\end{lemma}

\begin{proof}
For simplicity of notation, we assume the sign in \eqref{eq:lowOrder} is positive.
The other case is similar.

Let $G$ be the discrete monoid in $\mathbb{R}_{\ge 0}$ generated by $B-a-C$ and $C$.
Let $\eI_G \subset \Lambda_+$ be the following:
\begin{align}
 \eI_G=\left\{\sum_{i=0}^{\infty} a_i T^{\lambda_i} \in \Lambda_+ \  \Big | \  \lambda_i \in G\right \}.
\end{align}
Note that the coefficients of \eqref{eq:highG}
lie inside
\begin{align}
T^a \eI_G \cap T^B \eI_G. \label{eq:G1} 
\end{align}

For any $\bfb_1 \in \Lambda_+$,  by directly applying \eqref{eq:familiarW} (and noting that the disc classes contributing the first, second and final terms of \eqref{eq:smoothW} have intersection number with $[D_{S_1}]$ being $0$, $1$ and $1$, respectively), we have
\begin{align}
 W^{\bfb_1[D_{S_1}]}&=T^a(x_2^{-1}+y_2^{-1}+x_2+y_2)+e^{\bfb_1}T^B(x_1+y_1)+e^{\bfb_1}T^{B+C}(x_1^{-1}+y_1^{-1}) \label{eq:WW0}
\end{align}
Let $W_{\smooth}^{\bfb}$ denote those terms in $W^\bfb$ arising from contributions by smooth holomorphic discs, so $W_{\smooth}^{\bfb}=W^{\bfb_1[D_{S_1}]}$ (from our choice of $\bfb = \bfb_1 [D_{S_1}]+\mathbf{b}_{\orb} [\eX_1]$).

For $\bfb_1 \in \eI_G$, we have
\begin{align}
 W^\bfb=W_{\smooth}^{\bfb}+T^{B}(x_1y_1)^{-1}(x_2+y_2)+W_{hot} \label{eq:WW1}
\end{align}
where
$W_{hot}$ denotes the remaining `higher order terms' contributions.
Since $\bfb_1 \in \eI_G$ and using the observation that the coefficients of \eqref{eq:highG} lie inside
\eqref{eq:G1}, we know that 
all the coefficients of terms in $W_{hot}$ still lie inside \eqref{eq:G1}.

\begin{remark}
In the argument below, we could 
absorb the term $e^{\bfb_1}T^{B+C}(x_1^{-1}+y_1^{-1})$ in $W^{\bfb_1[D_{S_1}]}$
to $W_{hot}$; we do not do so to make it easier to track the explicit terms arising from the holomorphic curves considered previously. 
\end{remark}

The partial differentials of $W^\bfb$ are given by
 \begin{align}
  \partial_{x_2} W^\bfb=&T^a(-x_2^{-2}+1)+T^{B}(x_1y_1)^{-1}+ \partial_{x_2}W_{hot} \\
  \partial_{x_1} W^\bfb=&e^{\bfb_1}T^B+e^{\bfb_1}T^{B+C}(-x_1^{-2})+T^{B}(-x_1^{-2}y_1^{-1})(x_2+y_2) +\partial_{x_1}W_{hot}
 \end{align}
and similarly for $\partial_{y_2}W^\bfb, \partial_{y_1}W^\bfb$.

Since we are looking for solutions such that $x_1=y_1$ and $x_2=y_2$, by the symmetry between $x$ and $y$, it suffices to find $x_1,x_2 \in \Lambda_0 \setminus \Lambda_+$ simultaneously solving
\begin{align}
  0=&T^a(-x_2^{-2}+1)+T^{B}(x_1)^{-2}+ \partial_{x_2}W_{hot} \\
  0=&e^{\bfb_1}T^B+e^{\bfb_1}T^{B+C}(-x_1^{-2})+T^{B}(-x_1^{-3})(2x_2) +\partial_{x_1}W_{hot}
\end{align}
or equivalently 
\begin{align}
  0=&(-x_2^{-2}+1)+T^{B-a}(x_1)^{-2}+ T^{-a}\partial_{x_2}W_{hot} \label{eq:EE1}\\
  0=&e^{\bfb_1}+e^{\bfb_1}T^{C}(-x_1^{-2})+(-x_1^{-3})(2x_2) +T^{-B}\partial_{x_1}W_{hot}. \label{eq:EE2}
\end{align}
Note that  $T^{-a}\partial_{x_2}W_{hot}$ and $ T^{-B}\partial_{x_1}W_{hot}$ belong to $\eI_G$.

The leading order term equations are (cf. \eqref{eq:leading})
\begin{align}
 0=-x_2^{-2}+1 \quad \text{ and } \quad 0=1-x_1^{-3}(2x_2) \label{eq:leadingSolution}
\end{align}
which admit $6$ solutions in $(\mathbb{C}^*)^2$ given by $x_2=\pm 1$ and $(x_1)^{-3}= \pm \frac{1}{2}$.

We aim to solve  the equations \eqref{eq:EE1} and \eqref{eq:EE2} inductively.
We enumerate the elements in $G$ by $0<g_1<g_2< \dots$.

By \eqref{eq:leadingSolution}, the equations \eqref{eq:EE1} and \eqref{eq:EE2} are solvable modulo $T^{g_1}$.
Suppose that there is $\bfb_1=\bfb_{1,N} \in \eI_G$ such that 
\eqref{eq:EE1} and \eqref{eq:EE2} can be solved by $x_1=x_{1,N},x_2=x_{2,N} \in \eI_G$ modulo $T^{g_{N}}$.
It means that, modulo $T^{g_{N+1}}$ and after plugging in $x_1=x_{1,N},x_2=x_{2,N}$ and $\bfb_1=\bfb_{1,N}$, we have
\begin{align}
 c_1T^{g_N}=&(-x_2^{-2}+1)+T^{B-a}(x_1)^{-2}+ T^{-a}\partial_{x_2}W_{hot}  \label{eqCc1}\\
 c_2T^{g_N}=&e^{\bfb_1}+e^{\bfb_1}T^{C}(-x_1^{-2})+(-x_1^{-3})(2x_2) +T^{-B}\partial_{x_1}W_{hot} \label{eqCc2}
\end{align}
for some $c_1,c_2 \in \mathbb{C}$.
 
 Note that, for any $c \in \CC$, and modulo $T^{g_{N+1}}$, we have
 \begin{align}
  -(x_{2,N}+cT^{g_N})^{-2}&=-(x_{2,N})^{-2} \pm 2cT^{g_N}\\
  -x_{1,N}^{-3}(2(x_{2,N}+cT^{g_N}))&=-x_{1,N}^{-3}(2x_{2,N}) \mp cT^{g_N} \\
  e^{\bfb_{1,N}+cT^{g_N}}&=e^{\bfb_{1,N}}+cT^{g_N}
 \end{align}
because the leading terms of $x_{2,N}$ and $x_{1,N}^{-3}$ are $\pm 1$ and $ \pm \frac{1}{2}$, respectively.
 
 This means that if we define $x_{2,N+1}=x_{2,N} \pm \frac{c_1}{2}T^{g_N}$, 
 $x_{1,N+1}=x_{1,N}$ and $\bfb_{1,N+1}=\bfb_{1,N}+(c_2 + \frac{c_1}{2})T^{g_N}$,
 then we can solve the equations \eqref{eqCc1} and \eqref{eqCc2} modulo $T^{g_{N+1}}$, simultaneously.
 This completes the inductive step, and hence the proof.
 \end{proof}

\begin{remark}\label{r:sign}
Lemma \ref{l:lowCon} computed the lowest order contribution to 
$W^{\mathbf{b}_{\orb} [\eX_1]}$ arising from 
the Maslov $2$ holomorphic curves of Lemma \ref{l:preCl},  case $(2)$. Exactly two curves contribute. When working with a different trivialization of $T\eL$, it might be that the two curves contribute terms as in \eqref{eq:lowOrder} but with different signs.  Even in this case,  one can check that $W^{\bfb}$ will still have critical points for some appropriate $\bfb$.
 Recall that $D_{S_1}$ is the image of $S_1 \times M$ under the quotient map $M \times M \to \Sym^2(M)$.  The divisor $S_1$ has two connected components, given by the two irreducible components of the toric boundary of $M$ that have smaller area. One can `correct' for a difference in orientation signs for the holomorphic curves arising in the proof of Lemma \ref{l:lowCon} by taking the coefficients in $\bfb$ of the two connected components of $D_{S_1}$ to be different.
\end{remark}

\begin{remark}\label{r:Hilb2}
Let $\pi_{HC}: \Hilb^2(M) \to \eX$ be the Hilbert-Chow morphism, and let $D_{HC} \subset \Hilb^2(M)$ be the Hilbert-Chow divisor.  There is a K\"ahler form  $\omega_{\Hilb^2(M)}$
 that agrees with $\pi_{HC}^* \omega_{\eX}$ away from a small neighborhood of $D_{HC}$ (see \cite{Perutz1, Perutz2} for the construction). Since
the Lagrangian $\Sym(\eL)$ lies in the smooth locus of $\eX$, the preimage  
$\Sym(\eL)':=\pi_{HC}^{-1}(\Sym(\eL))$ is a Lagrangian submanifold  of the Hilbert scheme. We briefly explain why   $\Sym(\eL)'$ is non-displaceable for a particular symplectic form on the Hilbert scheme.
 
Let $J_{\eX,\Phi}$ be the symmetric product of $\Phi^*J_M$, and $J_{\Hilb^2(M),\Phi}$ be the complex structure on $\Hilb^2(M)$
such that $\pi_{HC}$ is $(J_{\eX,\Phi},J_{\Hilb^2(M),\Phi})$-holomorphic.
If $u':S \to \Hilb^2(M)$ is a $J_{\Hilb^2(M),\Phi}$-holomorphic disc with boundary on $\Sym(\eL)'$, then $u:=\pi_{HC} \circ u'$ is a 
$J_{\eX,\Phi}$-holomorphic disc with boundary on $\Sym(\eL)$.
We remark that when $u'$ is a stable map with sphere components, it is possible that spheres are mapped entirely into $D_{HC}$,
and that $u$ maps such sphere components entirely into the orbifold locus.
Conversely, if $u:S \to \eX$ is a $J_{\eX,\Phi}$-holomorphic disc with no sphere components lying entirely in the orbifold locus, then it 
hits the orbifold locus at finitely many points. This can be lifted uniquely to a $J_{\Hilb^2(M),\Phi}$-holomorphic disc $u':S \to \Hilb^2(M)$ such that $u'$ has no sphere components mapping entirely into $D_{HC}$.
In this case, the virtual dimensions of the moduli containing $u$ and $u'$ are the same. 

With this partial correspondence between holomorphic discs, we can partially compute the superpotential of $\Sym(\eL)'$
from the superpotential of $\Sym(\eL)$.
First of all, $\Sym(\eL)'$ does not bound any $J_{\Hilb^2(M),\Phi}$-holomorphic disc $u'$ with Maslov index $<2$, because otherwise $u=\pi_{HC} \circ u'$
would be a $J_{\eX,\Phi}$-holomorphic disc bound by $\Sym(\eL)$ with too small virtual dimension.
Secondly, the lower order terms of the superpotential of 
$\Sym(\eL)'$  and $\Sym(\eL)$ can be identified by the (partial) correspondence.
More precisely, a holomorphic disc $u$ with area $\Theta$ which hits the orbifold locus at finitely many points (with algebraic intersection $h$ in total) lifts to a holomorphic disc
$u'$ of area $\Theta-hr$ for some $r>0$ depending only on $\omega_{\Hilb^2(M)}$.
Since the lowest order Maslov $2$ holomorphic annuli \eqref{eq:lowOrder} correspond to holomorphic discs which have area $a+C$ and hit the orbifold locus at $2$ points, we need $2r=a+C-B$ for the lifted disc to have area $B$. 
On the other hand, $D_{S_1}$ lifts to a smooth divisor $D$ in $\Hilb^2(M)$.
When $2r=a+C-B$, we can take $\bfb=\bfb_1[D]$ to deform the superpotential of $\Sym(\eL)'$ 
in such a way that the 
leading term equations of the
bulk deformed superpotentials of $\Sym(\eL)'$  and $\Sym(\eL)$ agree.
The rest of the argument then follows on similar lines to that given previously.

\end{remark}


\end{document}